\documentclass{amsart}
\newif\ifhtml
\ifdefined\HCode
  \htmltrue
\else
  \htmlfalse
\fi
\numberwithin{equation}{section}

\usepackage{amsmath,amssymb}
\ifhtml
\else
  \usepackage[foot]{amsaddr}
\fi
\usepackage{graphicx,tabularx,xcolor,array,url,aliascnt}
\usepackage{enumitem}
\usepackage[initials,short-journals,short-months]{amsrefs}

\ifhtml
  \usepackage[colorlinks=true]{hyperref}
\else
  \usepackage{zref-clever}
  \zcsetup{cap}
  \usepackage[pdfstartview={FitH},colorlinks=true]{hyperref}
  \hypersetup{hypertexnames=false}
  \usepackage{tikz}
  \usetikzlibrary{graphs,quotes,shapes.geometric}
\fi

\renewcommand{\arraystretch}{1.1}

\theoremstyle{plain}
\newtheorem{theorem}{Theorem}[section]
\newaliascnt{lemma}{theorem}
\newtheorem{lemma}[lemma]{Lemma}
\aliascntresetthe{lemma}
\newaliascnt{proposition}{theorem}
\newtheorem{proposition}[proposition]{Proposition}
\aliascntresetthe{proposition}
\newaliascnt{corollary}{theorem}
\newtheorem{corollary}[corollary]{Corollary}
\aliascntresetthe{corollary}
\newtheorem*{scheme}{Octopus induction scheme}

\theoremstyle{remark}
\newaliascnt{remark}{theorem}
\newtheorem{remark}[remark]{Remark}
\aliascntresetthe{remark}
\newaliascnt{example}{theorem}
\newtheorem{example}[example]{Example}
\aliascntresetthe{example}

\ifhtml
  \renewcommand{\address}[1]{}
  \renewcommand{\email}[1]{}

  \newcommand{\zcref}[1]{\autoref{#1}}
\fi

\ifhtml
  \newcommand{\GraphFourCycle}{\HCode{<img class="inline-graph graph-small" src="html-assets/four-cycle.svg" alt="Four-cycle with vertices 4, 2, 3, 1 in cyclic order" />}}
  \newcommand{\GraphFiveVertexException}{\HCode{<img class="inline-graph graph-medium" src="html-assets/five-vertex-exception.svg" alt="Exceptional five-vertex family: edges 12, 23, and 14 have common weight w greater than zero; edges 35 and 45 have weights q and r with q plus r greater than zero; edge 34 has weight w minus qr divided by q plus r, which is nonnegative; all other weights are zero" />}}
  \newcommand{\GraphReductionDisplay}{\HCode{<div class="graph-reduction"><span class="graph-label">G =</span><img src="html-assets/five-cycle.svg" alt="Weighted five-cycle G" /><span class="graph-arrow">reduce at vertex 5 →</span><span class="graph-label">H =</span><img src="html-assets/reduced-four-cycle.svg" alt="Reduced weighted four-cycle H" /></div>}}
\else
  \newcommand{\GraphFourCycle}{%
    \begin{tikzpicture}[baseline]{
    \draw
        (-0.5,-0.5) node[anchor=base,draw,circle] (A) {$4$}
        (-0.5,0.5) node[anchor=base,draw,circle] (B) {$2$}
        (0.5,0.5) node[anchor=base,draw,circle] (C) {$3$}
        (0.5,-0.5) node[anchor=base,draw,circle] (D) {$1$};
    \draw (D)--(A); \draw (A)--(B); \draw (B)--(C); \draw (C)--(D);
    }\end{tikzpicture}}
  \newcommand{\GraphFiveVertexException}{%
    \begin{tikzpicture}[baseline]{
    \draw
        (-0.5,-0.5) node[anchor=base,draw,circle] (A) {$1$}
        (-0.5,0.5) node[anchor=base,draw,circle] (B) {$2$}
        (0.5,0.5) node[anchor=base,draw,circle] (C) {$3$}
        (0.5,-0.5) node[anchor=base,draw,circle] (D) {$4$}
        (1.5,0) node[anchor=base,draw,circle] (E) {$5$};
    \draw (D)--(A)--(B)--(C);
    \draw[dashed] (C)--(D); \draw[dashed] (C)--(E)--(D);
    }\end{tikzpicture}}
  \newcommand{\GraphReductionDisplay}{%
    \[
    G=\left(\begin{tikzpicture}[baseline]{
    \draw
        (-0.5,-0.5) node[anchor=base,draw,circle] (A) {$4$}
        (-0.5,0.5) node[anchor=base,draw,circle] (B) {$3$}
        (0.5,0.5) node[anchor=base,draw,circle] (C) {$2$}
        (0.5,-0.5) node[anchor=base,draw,circle] (D) {$1$}
        (1.5,0) node[anchor=base,draw,circle] (E) {$5$};
    \draw (D)--(A) node[midway,below=5pt] {$c_{14}$};
    \draw (A)--(B) node[midway,left=5pt] {$c_{34}$};
    \draw (B)--(C) node[midway,above=5pt] {$c_{23}$};
    \draw (C)--(E) node[near start,above right] {$c_{25}$};
    \draw (D)--(E) node[near start,below right] {$c_{15}$};
    }\end{tikzpicture}\right)
    \quad\xrightarrow[\text{vertex $5$}]{\text{reduce at}}\quad
    H=\left(\begin{tikzpicture}[baseline]{
    \draw
        (-0.5,-0.5) node[anchor=base,draw,circle] (A) {$4$}
        (-0.5,0.5) node[anchor=base,draw,circle] (B) {$3$}
        (0.5,0.5) node[anchor=base,draw,circle] (C) {$2$}
        (0.5,-0.5) node[anchor=base,draw,circle] (D) {$1$};
    \draw (D)--(A) node[midway,below=5pt] {$c_{14}$};
    \draw (A)--(B) node[midway,left=5pt] {$c_{34}$};
    \draw (B)--(C) node[midway,above=5pt] {$c_{23}$};
    \draw (C)--(D) node[midway,right=5pt] {$\tilde c_{12}=\frac{c_{15}c_{25}}{c_{15}+c_{25}}$};
    }\end{tikzpicture}\right)
    \]}
\fi

\title[Second eigenspace of the interchange process]{Uniqueness of the second eigenspace of the interchange process}
\author{Dennis Belotserkovskiy and Joe P.\@ Chen}
\address{Department of Mathematics, Colgate University, Hamilton, NY 13346, USA}
\email{dbelotserkovskiy@colgate.edu, jpchen@colgate.edu}
\date{\today}
\thanks{DB thanks the Division of Natural Sciences \& Mathematics at Colgate University for financial support during the summer 2024 research experience. JPC thanks the Research Council of Colgate University for partial financial support.}

\begin{document}

\begin{abstract}
The spectral gap theorem of Caputo, Liggett, and Richthammer states that on any connected weighted graph, the second eigenvalue of the interchange process equals the second eigenvalue of the random walk process.
We characterize the corresponding eigenspace in the regular representation of the symmetric group.
Except when the graph is a $4$-cycle whose four edge weights are equal, the entire second eigenspace lies in the direct sum of copies of the standard representation and is generated by copies of the second eigenspace of the random-walk Laplacian.
The proof refines the octopus induction scheme by isolating the summand of the restricted representation that can support equality, and analyzing it in explicit two-subset and exterior-square coordinates.
\end{abstract}

\maketitle
\ifhtml
\begin{center}
\small Department of Mathematics, Colgate University, Hamilton, NY 13346, USA\\
\href{mailto:dbelotserkovskiy@colgate.edu}{dbelotserkovskiy@colgate.edu}
\qquad
\href{mailto:jpchen@colgate.edu}{jpchen@colgate.edu}
\end{center}
\fi

\setcounter{tocdepth}{1}
\tableofcontents

\section{Introduction}
\label{sec:intro}

\subsection{A recursive construction and an exceptional example}
\label{sec:begin}

On a connected $n$-vertex undirected graph $G=(V(G),E(G))$, and $1\leq k\leq \lfloor \frac{n}{2}\rfloor$, consider the \emph{$k$-particle exclusion process} on $G$ with rates $c_{ij}>0$, $ij\in E(G)$.
Its state space is $\binom{V(G)}{k}$, the set of $k$-subsets of $V(G)$.
Throughout the paper we use the positive-Laplacian convention: $A^{(k)}$ denotes the positive semidefinite Markov Laplacian, so the probabilistic infinitesimal generator is $-A^{(k)}$.  Its matrix is
\begin{align}
\label{eq:Ak}
(A^{(k)})_{\Omega, \Lambda} =
\left\{
\begin{array}{ll}
-c_{ij}, & \text{if } \Lambda = (\Omega \sqcup \{j\}) \setminus \{i\} \text{ for some } i\in \Omega,~ j\notin\Omega,\\
\sum_{i\in \Omega} \sum_{j\notin\Omega}c_{ij}, & \text{if } \Lambda=\Omega,\\
0, & \text{otherwise},
\end{array}
\right.
\qquad
\Omega, \Lambda \in \binom{V(G)}{k}.
\end{align}
The matrix $A^{(k)}$ is symmetric and positive semidefinite, and the exclusion chain is irreducible because $G$ is connected. 
Every row and column of $A^{(k)}$ sums to $0$, so $\sum_{\Omega\in \binom{V(G)}{k}} \delta_\Omega$ is a $0$-eigenvector.  Irreducibility makes this eigenspace one-dimensional.
Also, when $k=1$, we recover the \emph{random walk process} on $G$, and $A^{(1)}$ is the graph Laplacian.
Let the eigenvalues of $A^{(k)}$ be listed in increasing order, $0=\lambda_1(A^{(k)}) < \lambda_2(A^{(k)}) \leq \cdots \leq \lambda_{\binom{n}{k}}(A^{(k)})$.

According to the spectral gap theorem of Caputo, Liggett, and Richthammer \cite{CLR}, to be described shortly, $\lambda_2(A^{(k)})$ is identical for every $1\leq k\leq \lfloor \frac{n}{2}\rfloor$.
What about the second eigenvectors of $A^{(k)}$?
We first recall the standard raising construction for eigenvectors of $A^{(k)}$.

\begin{proposition}
\label{prop:eigrec}
For $1\leq k \leq \lfloor \frac{n}{2}\rfloor-1$ and $u\in \mathbb{R}^{\binom{V(G)}{k}}$, define $u^\uparrow \in \mathbb{R}^{\binom{V(G)}{k+1}}$ by
\begin{align}
\label{eq:uup}
(u^\uparrow )_\Omega = \sum_{i\in \Omega} u_{\Omega \setminus \{i\}}, \qquad \Omega \in \binom{V(G)}{k+1}.
\end{align}
If $A^{(k)} u = \lambda u$ for some eigenvalue $\lambda \in \mathbb{R}$, then $A^{(k+1)} u^\uparrow = \lambda u^\uparrow$.
\end{proposition}
\begin{proof}
From \eqref{eq:Ak} one finds
\begin{align}
\label{eq:Aku}
(A^{(k)} u)_{\Omega} = \sum_{i\in \Omega} \sum_{j\notin \Omega} c_{ij} \left[u_\Omega- u_{(\Omega \sqcup \{j\}) \setminus \{i\}}\right], \quad \Omega \in \binom{V(G)}{k}.
\end{align}
Similarly
\begin{align}
\label{eq:Akup}
(A^{(k+1)} u^\uparrow)_{\Omega} &=
\sum_{i\in \Omega} \sum_{j\notin \Omega} c_{ij} \left[(u^\uparrow)_\Omega - (u^\uparrow)_{(\Omega\sqcup\{j\})\setminus \{i\}}\right], \quad \Omega \in \binom{V(G)}{k+1}.
\end{align}
Now
\begin{align*}
(u^\uparrow)_\Omega - (u^\uparrow)_{(\Omega\sqcup\{j\})\setminus \{i\}}
&=
\sum_{\ell\in \Omega} u_{\Omega \setminus \{\ell\}}
-
\sum_{\ell'\in (\Omega\sqcup \{j\})\setminus \{i\}} u_{(\Omega \sqcup \{j\})\setminus \{i,\ell'\}} \\
&=
\sum_{\ell\in \Omega \setminus \{i\}} u_{\Omega \setminus \{\ell\}} - \sum_{\ell' \in \Omega \setminus \{i\}}  u_{(\Omega \sqcup \{j\})\setminus \{i,\ell'\}}
\end{align*}
using the cancellation occurring at $(\ell,\ell')=(i,j)$. Plugging this into \eqref{eq:Akup} yields
\begin{align*}
(A^{(k+1)} u^\uparrow)_{\Omega} 
&=
\sum_{i\in \Omega} \sum_{j\notin \Omega} c_{ij} \sum_{\ell\in \Omega \setminus \{i\}} \left[u_{\Omega\setminus \{\ell\}} - u_{(\Omega \sqcup \{j\})\setminus \{i,\ell\}}\right]\\
&=
\sum_{\ell\in\Omega} \sum_{i\in \Omega \setminus \{\ell\}} \sum_{j\notin \Omega} c_{ij}  \left[u_{\Omega\setminus \{\ell\}} - u_{(\Omega \sqcup \{j\})\setminus \{i,\ell\}}\right].
\end{align*}
For a fixed $\Omega$, we may enlarge the innermost sum from $j\notin\Omega$ to $j\notin\Omega\setminus\{\ell\}$ by adding the terms with $j=\ell$.  Their total contribution is
\[
\sum_{\ell\in\Omega}\sum_{i\in\Omega\setminus\{\ell\}}
 c_{i\ell}\bigl(u_{\Omega\setminus\{\ell\}}-u_{\Omega\setminus\{i\}}\bigr)=0,
\]
because the summands cancel in pairs under $(i,\ell)\leftrightarrow(\ell,i)$.  Hence, after invoking \eqref{eq:Aku},
\[
(A^{(k+1)}u^\uparrow)_\Omega
=\sum_{\ell\in\Omega}\sum_{i\in\Omega\setminus\{\ell\}}
 \sum_{j\notin\Omega\setminus\{\ell\}}c_{ij}
 \left[u_{\Omega\setminus\{\ell\}}-
 u_{((\Omega\setminus\{\ell\})\sqcup\{j\})\setminus\{i\}}\right]
=\sum_{\ell\in\Omega}(A^{(k)}u)_{\Omega\setminus\{\ell\}}.
\]
Now apply the hypothesis $A^{(k)}u =\lambda u$ and \eqref{eq:uup} to the last display to obtain the desired identity $A^{(k+1)}u^\uparrow = \lambda u^\uparrow$.
\end{proof}

Via \zcref{prop:eigrec}, every nonzero raised vector obtained from a second eigenvector of the graph Laplacian $A^{(1)}$ is a second eigenvector of $A^{(2)}$, and the construction can then be iterated.
But does one get \emph{all} the second eigenvectors of $A^{(k)}$?
The following example says no.

\begin{example}
\label{ex:4cycle}
Let $G$ be the $4$-cycle
\GraphFourCycle
with all four edge weights equal to $1$; explicitly,
$c_{13}=c_{23}=c_{14}=c_{24}=1$ and $c_{12}=c_{34}=0$.
The graph Laplacian is
\begin{align*}
    A^{(1)}=\begin{bmatrix}
        2 & 0& -1 & -1 \\
        0 & 2 & -1 & -1 \\
        -1 & -1 & 2 & 0 \\
        -1 & -1 & 0 & 2
    \end{bmatrix},
\end{align*}
which has eigenvalues $0, 2,2 ,4$.
The second eigenspace of $A^{(1)}$ is the span of
\begin{align}
\label{eq:A1eig}
[1, -1, 0, 0]^{\mathsf T} \quad \text{and} \quad [0,0,1,-1]^{\mathsf T}.
\end{align}
Meanwhile, if we list the 2-subsets of $\{1,2,3,4\}$ in the order $(12), (13), (14), (23), (24), (34)$, then the positive Laplacian of the 2-particle exclusion process is
\begin{align*}
    A^{(2)}=\begin{bmatrix}
        4 & -1&-1&-1&-1& 0\\
        -1&2&0&0&0& -1\\
        -1&0&2&0&0&-1\\
        -1&0&0&2&0&-1\\
        -1&0&0&0&2&-1\\
        0&-1&-1&-1&-1&4
    \end{bmatrix},
\end{align*}
which has eigenvalues $0,2,2,2,4,6$.
The second eigenspace of $A^{(2)}$ is the span of
\begin{align}
\label{eq:A2eig}
\boldsymbol{w}_1=[0,-1,1, 0, 0, 0]^{\mathsf T}, \quad
\boldsymbol{w}_2=[0, -1,0,1,0,0]^{\mathsf T},\quad
\boldsymbol{w}_3=[0, -1 ,0,0, 1, 0]^{\mathsf T}.
\end{align}
Applying \zcref{prop:eigrec} to \eqref{eq:A1eig} generates two second eigenvectors of $A^{(2)}$,
\begin{align}
\label{eq:A1eigA2}
\boldsymbol{w}_1 - \boldsymbol{w}_2 -\boldsymbol{w}_3 = [0, 1,1,-1,-1, 0]^{\mathsf T} \quad\text{and}\quad
-\boldsymbol{w}_1 + \boldsymbol{w}_2 -\boldsymbol{w}_3 = [0, 1, -1, 1 ,-1 ,0]^{\mathsf T}.
\end{align}
This means a third second eigenvector of $A^{(2)}$,
\begin{align}
\label{eq:3rdA2eig}
-\boldsymbol{w}_1  -\boldsymbol{w}_2 + \boldsymbol{w}_3 = [0, 1,-1,-1,1,0]^{\mathsf T},
\end{align}
is not accounted for by the second eigenvectors of $A^{(1)}$.
\end{example}

The additional vector in \zcref{ex:4cycle} raises the structural question that drives the manuscript: might there exist another connected graph for which the second eigenspace of $A^{(k)}$ is not generated by the second eigenvectors of $A^{(1)}$?

The present paper answers this question completely.
It turns out that this property resides in a higher process called the \emph{interchange process}, a continuous-time Markov chain on the symmetric group $\mathfrak{S}_n$ with transpositions being the allowed transitions.
In the same positive-Laplacian convention, the interchange Laplacian has matrix
\[
(L_G)_{g,g'} = 
\left\{
\begin{array}{ll}
-c_{ij}, &\text{if } g'=g(i,j),\\
\sum_{1\leq i<j\leq n} c_{ij}, & \text{if } g'=g,\\
0, &\text{otherwise},
\end{array}
\right.
\qquad g,g'\in \mathfrak{S}_n,
\]
where $(i,j)$ denotes the transposition interchanging $i<j$ in $[n]:=\{1,\ldots,n\}$, and $c_{ij}=0$ for nonedges.  Thus $-L_G$ is the probabilistic infinitesimal generator.
Our main \zcref{thm:main} characterizes the second eigenspace of $L_G$ on every finite connected weighted graph.
As a corollary, we will determine the multiplicity of the second eigenvalue $\lambda_2(A^{(k)})$.

\subsection{Known complete-graph eigenspace results}

For the complete graph $K_n$ with every edge weight equal to $1$, the spectra of both $A^{(k)}$ and $L_G$ are explicitly known.

For $2\leq k\leq \lfloor \frac{n}{2}\rfloor$, the $k$-particle exclusion process on $K_n$ is the random walk on the \emph{Johnson graph} $J(n,k)$: its vertices are the $k$-subsets of $[n]$, and two vertices are adjacent when the corresponding subsets differ by exchanging one element.
The graph Laplacian on $J(n,k)$ has eigenvalues $\lambda^{(i)}=i(n+1-i)$ for $i\in \{0,1,\cdots, k\}$, with corresponding multiplicity $\binom{n}{i}-\binom{n}{i-1}$. 
Under the natural symmetric-group action, the $\lambda^{(i)}$-eigenspace is the irreducible summand $S^{(n-i,i)}$ inside the permutation module on the $k$-subsets of $[n]$. (We summarize the representation theory of the symmetric group in \zcref{sec:setup}.)
For an accessible proof without representation theory, see \cite{GodsilMeagher}*{Sections 6.2$\sim$6.3}.

For the interchange process on $K_n$, Diaconis and Shahshahani \cite{DS81}*{Corollary 4} identified all eigenvalues and their multiplicities.  Because the operator is central in this case, it acts by a scalar on each irreducible summand of the regular representation; a full eigenspace may be the direct sum of several such summands when their scalars coincide. 

On a weighted graph, we note that \cite{CLR}*{Section 1.4} lists some eigenvalues of $L_G$, along with their multiplicities, that arise from certain irreducible representations of $\mathfrak{S}_n$.

\subsection{The spectral gap theorem revisited}
\label{sec:history}

The second eigenvalue, or \emph{spectral gap}, problem has received much attention in the past three decades.
Since the random walk process is a projection of the interchange process, $\lambda_2(L_G) \leq \lambda_2(A^{(1)})$. 
Around 1992 Aldous conjectured that on every graph whose present edges have weight $1$, $\lambda_2(L_G)=\lambda_2(A^{(1)})$.
Many partial results followed in the years since, but it was Caputo, Liggett, and Richthammer in 2010 \cite{CLR} who decisively established the equality on all weighted graphs, using a Schur reduction scheme on the number of vertices that was inspired by the recursive approach of Handjani and Jungreis \cite{HJ}.
We shall refer to the scheme, which consists of three inequalities and one equality, as the \emph{octopus induction scheme} in \zcref{sec:octopus} below.
See \cite{CLR}*{Section 1.3} for a history behind the spectral gap theorem.

Our focus is on the second eigenspace of $L_G$, which appears to receive less attention.
To identify the second eigenspace, we revisit the octopus induction scheme and determine when each inequality in that scheme can be an equality.
In particular, we need to understand when equality in the \emph{octopus inequality} \cite{CLR}*{Theorem 2.3} holds in various settings.

\section{Setup and main results}
\label{sec:setup}

Throughout this paper, $G$ denotes a graph on $[n]:=\{1,\ldots,n\}$ equipped with symmetric weights $c_{ij}=c_{ji}\geq0$ for $i\neq j$, with $c_{ii}=0$.
Let
\[
E:=\{ij:1\leq i<j\leq n,\ c_{ij}>0\}.
\]
We call the unweighted graph $([n],E)$ the \emph{support graph}, or \emph{contact graph}, of the weights.  Graph-theoretic terms such as connectedness, degree, path, and cycle always refer to this graph.  If there exists $c>0$ such that $c_{ij}=c$ (resp.\@ $c_{ij}=1$) for every $ij\in E$, we say that $G$ has \emph{uniform} (resp.\@ \emph{simple}) weights.
Finally, for a nonempty subset $\Omega\subseteq[n]$, let $\mathfrak{S}_\Omega$ denote the group of permutations of $\Omega$.

\subsection{Operators on the symmetric group algebra}

Let $\mathfrak{S}_n$ denote the symmetric group on $[n]$.
A real $\mathfrak{S}_n$-module is a real vector space equipped with a linear left action of $\mathfrak{S}_n$.
The basic example is the \emph{group algebra}
\[
\mathbb{R}[\mathfrak{S}_n]
:=\left\{\sum_{g\in\mathfrak{S}_n}\gamma_g g:\gamma_g\in\mathbb{R}\right\},
\]
the vector space spanned by the elements of $\mathfrak{S}_n$, equipped here with the left regular action.
We equip $\mathbb{R}[\mathfrak{S}_n]$ with the inner product $\langle\cdot,\cdot\rangle_{\mathfrak{S}_n}$ for which the permutations form an orthonormal basis.  The inversion map $g\mapsto g^{-1}$ identifies the right-regular matrix in the introduction, whose transitions are $g\mapsto g(i,j)$, with the left-regular action used below.  Thus both realizations have the same spectrum and eigenspace multiplicities.

We introduce three closely related operators $L_G, L_H, \Delta$ on the symmetric group algebra. Each of them is a weighted sum of the group algebra elements $\operatorname{Id}-(i,j)$, where $(i,j)$ is the transposition between $i<j$.
The main operator of interest is the \emph{interchange operator} on $G$,
\begin{align}
\label{eq:LG}
L_G := \sum_{1\leq i<j\leq n} {c}_{ij} (\operatorname{Id}- (i,j)) \quad\text{ on } \mathbb{R}[\mathfrak{S}_n].
\end{align}
Directly from the definition, $L_G$ is a symmetric, positive semidefinite operator with respect to the inner product $\langle\cdot,\cdot\rangle_{\mathfrak{S}_n}$, and $L_G \left(\sum_{g\in \mathfrak{S}_n} g\right)= 0$.
If $G$ is connected, then the kernel of $L_G$ is $1$-dimensional, equaling the span of $\sum_{g\in\mathfrak{S}_n} g$.

Next, we reduce the graph $G$ at vertex $n$, and produce a new graph $H$ on $[n-1]$ with edge weights 
\begin{align}
\label{eq:redweights}
\tilde{c}_{ij} :={c}_{ij} + \frac{{c}_{in} {c}_{jn}}{\sum_{k=1}^{n-1} {c}_{kn}}, \quad 1\leq i<j\leq n-1.
\end{align}
Formula \eqref{eq:redweights} is obtained via a Schur complement computation which is explained in \zcref{sec:Schur} below. 
If $G$ is connected, then $H$ is connected: along any path in $G$, an occurrence of the removed vertex $n$ between two neighbors $i$ and $j$ can be replaced by the positive reduced edge of weight $c_{in}c_{jn}/\sum_k c_{kn}$.
We call this operation the \emph{Schur reduction}, or \emph{Kron reduction}, of $G$ at $n$.
Analogously we define the interchange operator on $H$ as
\[
L_H := \sum_{1\leq i<j\leq n-1} \tilde{c}_{ij} (\operatorname{Id}-(i,j)) \quad\text{ on } \mathbb{R}[\mathfrak{S}_{n-1}].
\]
Write $H\sqcup\{n\}$ for the graph obtained by adjoining $n$ to $H$ as an isolated vertex.  Then $L_H$ extends naturally to an operator $L_{H\sqcup\{n\}}$ on $\mathbb{R}[\mathfrak{S}_n]$, with $n$ fixed by every permutation in $\mathfrak{S}_{n-1}$.

Given $G$ and its reduced graph $H$, we define the corresponding \emph{octopus operator} on $\mathbb{R}[\mathfrak{S}_n]$,
\begin{align}
\label{eq:octopusop}
\Delta := L_G - L_{H\sqcup \{n\}} = \sum_{i=1}^{n-1} {c}_{in}(\operatorname{Id}-(i,n))  - \sum_{1\leq i<j\leq n-1} \frac{{c}_{in} {c}_{jn}}{\sum_{k=1}^{n-1} {c}_{kn}} (\operatorname{Id} - (i,j)).
\end{align}
The significance of $\Delta$ comes from the nontrivial \emph{octopus inequality} of Caputo, Liggett, and Richthammer, which states that $\Delta$ is a positive semidefinite operator on $\mathbb{R}[\mathfrak{S}_n]$ \cite{CLR}*{Theorem 2.3}.
See \zcref{sec:octopus} for the role $\Delta$ plays in the spectral gap induction, and \zcref{sec:rigidity} for the branchwise analysis of its equality cases.

\subsection{Symmetric group and Young tableaux}
\label{sec:sgyt}

A \emph{submodule} of a $\mathfrak{S}_n$-module is an invariant linear subspace.  A nonzero module is \emph{irreducible} if it has no proper nonzero submodule, and \emph{reducible} otherwise.
By Maschke's theorem \cite{Sagan}*{\S 1.5}, every finite-dimensional $\mathfrak{S}_n$-module over $\mathbb{R}$ is a direct sum of irreducible submodules.

We summarize the required \emph{Young tableau} machinery below; see \cite{Sagan}*{Chapter 2} for further details.

We say that $\mu=(\mu_1, \mu_2, \cdots, \mu_q)$ is a \emph{partition} of $n$, denoted $\mu \vdash n$, if the positive integers $\mu_1\geq \mu_2 \geq \cdots \geq \mu_q$ satisfy $\sum_{i=1}^q \mu_i=n$.
Each partition $\mu\vdash n$ is represented by a \emph{Young diagram}, an array of $n$ boxes having $q$ left-justified rows, with row $i$ containing $\mu_i$ boxes.
If $\mu$ and $\lambda=(\lambda_1,\ldots,\lambda_r)$ are partitions of $n$, we say that $\mu$ \emph{dominates} $\lambda$, denoted $\mu\unrhd\lambda$, when
$\mu_1+\cdots+\mu_i\geq\lambda_1+\cdots+\lambda_i$ for every $i\geq1$.
A partition with at most two nonzero parts is called \emph{two-row}, and a partition of the form $(n-r,1^r)$ is called a \emph{hook}.

Observe that there is a natural association between $\mu$ and the Young subgroup $\mathfrak{S}_\mu:=\mathfrak{S}_{\{1,\cdots, \mu_1\}} \times \mathfrak{S}_{\{\mu_1+1,\cdots, \mu_1+\mu_2\}} \times \cdots \times \mathfrak{S}_{\{\mu_1+\cdots+\mu_{q-1}+1,\cdots, n\}}$ of $\mathfrak{S}_n$.
Tableaux and their row-equivalence classes provide a concrete model for the cosets of $\mathfrak{S}_\mu$ in $\mathfrak{S}_n$.
To be precise, each $\mu$-tableau is an array $t$ obtained by filling the boxes of the Young diagram of shape $\mu$ with the positive integers $1, 2,\cdots, n$ bijectively.
For instance, there are six $(2,1)$-tableaux:
\[
\begin{smallmatrix}\boxed{1}&\boxed{2}\\[-1pt]\boxed{3}&\end{smallmatrix}
~,\quad
\begin{smallmatrix}\boxed{ 2}&\boxed{ 1}\\[-1pt]\boxed{ 3}&\end{smallmatrix}
~,\quad
\begin{smallmatrix}\boxed{ 1}&\boxed{ 3}\\[-1pt]\boxed{ 2}&\end{smallmatrix}
~,\quad
\begin{smallmatrix}\boxed{ 3}&\boxed{ 1}\\[-1pt]\boxed{ 2}&\end{smallmatrix}
~,\quad
\begin{smallmatrix}\boxed{ 2}&\boxed{ 3}\\[-1pt]\boxed{ 1}&\end{smallmatrix}
~,\quad
\begin{smallmatrix}\boxed{ 3}&\boxed{ 2}\\[-1pt]\boxed{ 1}&\end{smallmatrix}
~.
\]
We will be particularly interested in the \emph{standard} tableaux, wherein the rows and columns form increasing sequences.
In the preceding example, only 
\(\begin{smallmatrix}\boxed{ 1}&\boxed{ 2}\\[-1pt]\boxed{ 3}&\end{smallmatrix}\)
and
\(\begin{smallmatrix}\boxed{ 1}&\boxed{ 3}\\[-1pt]\boxed{ 2}&\end{smallmatrix}\)
are standard $(2,1)$-tableaux.

Given a tableau $t$ we can produce isomorphic copies of certain Young subgroups.
Let $R_1, \cdots, R_q$ (resp.\@ $C_1, \cdots, C_m$) denote the rows (resp.\@ columns) of $t$.
The \emph{row} (resp.\@ \emph{column}) \emph{stabilizer group} of $t$ is the subgroup $\mathfrak{R}_t = \mathfrak{S}_{R_1} \times \cdots \times \mathfrak{S}_{R_q}$ (resp.\@ $\mathfrak{C}_t = \mathfrak{S}_{C_1} \times \cdots \times \mathfrak{S}_{C_m}$) of $\mathfrak{S}_n$.

The action of $g\in\mathfrak{S}_n$ on a $\mu$-tableau $t$ is obtained by applying $g$ to every entry of $t$.

\subsection{Permutation modules and Specht modules}
\label{sec:permSpecht}

Two $\mu$-tableaux $t$ and $t'$ are said to be \emph{row-equivalent} if there exists $\sigma \in \mathfrak{R}_t$ such that $t'= \sigma t$, \emph{i.e.,} both $t$ and $t'$ have the same entries in every row. 
This defines an equivalence relation $\sim$ on the tableaux, and we call each row-equivalence class $\{t\} :=\{ t': t'\sim t\}$ a \emph{tabloid}.
In the preceding example there are three $(2,1)$-tabloids:
\[
\left\{\begin{matrix}1\;2\\3\end{matrix}\right\}
~,\quad
\left\{\begin{matrix}1\;3\\2\end{matrix}\right\}
~,\quad
\left\{\begin{matrix}2\;3\\1\end{matrix}\right\}
~,
\]
where the first two tabloids are represented by standard tableaux.
Note that $g\{t\} = \{gt\}$ for every $g\in \mathfrak{S}_n$.
If $\{t_1\}, \cdots, \{t_k\}$ is a complete list of $\mu$-tabloids---there are $n!/\prod_{i=1}^q \mu_i!$ of them---then the vector space 
\[
M^\mu = \operatorname{span}_{\mathbb{R}}\{\{t_1\},\ldots,\{t_k\}\}
\]
equipped with the left action described above forms the \emph{permutation module} corresponding to $\mu$.
Equivalently, $M^\mu$ is the orbit span of any one tabloid: $M^\mu=\operatorname{span}_{\mathbb{R}}\{g\{t\}:g\in\mathfrak{S}_n\}$.

Some important permutation modules are: the \emph{trivial representation} $M^{(n)}$; the \emph{defining representation} $M^{(n-1,1)}$; and the \emph{regular representation} $M^{(1^n)}$.
Consider the action of $L_G$ \eqref{eq:LG} on each of these modules.
First of all, $L_G M^{(n)}=0$ as expected.
Next, each tabloid in $M^{(1^n)}$ can be identified with a unique $g\in\mathfrak{S}_n$ by reading its singleton rows as the one-line word $g(1)\cdots g(n)$.
So $M^{(1^n)} \cong \mathbb{R}[\mathfrak{S}_n]$, and the action of $L_G$ on $M^{(1^n)}$ is the positive Laplacian of the interchange process on $G$.
Finally, for each $\ell\in [n]$, the tabloid $\{t_\ell\}$ in $M^{(n-1,1)}$ with entry $\ell$ in row $2$ can be identified with the unit coordinate vector $\boldsymbol{e}_\ell$ of $\mathbb{R}^n$.
Since 
\[
(i,j) \{t_\ell\} = \left\{
\begin{array}{ll}
\{t_j\}, & \ell=i,\\
\{t_i\}, & \ell=j,\\
\{t_\ell\}, & \text{otherwise},
\end{array}
\right.
\]
we deduce that the action of $L_G$ on $M^{(n-1,1)}$ is isomorphic to the action of the \emph{graph Laplacian} $\mathcal{L}_G$ on $\mathbb{R}^n$, with corresponding matrix
\begin{align}
\label{eq:graphLap}
\mathcal{L}_G = 
\begin{bmatrix}
\sum_{k\neq 1} {c}_{1k} & -{c}_{12} & -{c}_{13} &\cdots & -{c}_{1n} \\
-{c}_{12} & \sum_{k\neq 2}{c}_{2k} & -{c}_{23} & \cdots & -{c}_{2n}\\
-{c}_{13} &-{c}_{23} &  \sum_{k\neq 3}{c}_{3k}   & \cdots & -{c}_{3n}\\
\vdots & \vdots & \vdots& \ddots & \vdots \\
-{c}_{1n} & -{c}_{2n} & -{c}_{3n} & \cdots & \sum_{k\neq n} {c}_{nk}
\end{bmatrix}
\end{align}
in the coordinate  basis $\{\boldsymbol{e}_1, \cdots, \boldsymbol{e}_n\}$.
Thus the action of $L_G$ on $M^{(n-1,1)}$ is the positive Laplacian of the random walk on $G$.
Note that $\mathcal{L}_G$ is positive semidefinite with eigenvalues $0=\lambda_1(\mathcal{L}_G) \leq \lambda_2(\mathcal{L}_G) \leq \cdots \leq \lambda_n(\mathcal{L}_G)$.
If $G$ is connected, then $\lambda_2(\mathcal{L}_G)>0$ by the Perron-Frobenius theorem.


Except for the trivial representation, the permutation modules $M^\mu$ are reducible, so we next identify the irreducible representations of $\mathfrak{S}_n$.
Again let $t$ be a $\mu$-tableau, $\mathfrak{C}_t$ be its column group, and $\mathfrak{C}_t^- := \sum_{\sigma \in \mathfrak{C}_t} \operatorname{sgn}(\sigma) \sigma$ be the signed group sum of $\mathfrak{C}_t$.
The \emph{polytabloid} associated with the tableau $t$ is
\begin{align}
\label{eq:etdef}
\boldsymbol{e}_t := \mathfrak{C}_t^- \{t\} \in M^\mu.
\end{align}
Returning to the example, the three $(2,1)$-polytabloids are
\[
\begin{matrix}1&2\\3&\end{matrix}
=
\left\{\begin{matrix}1\;2\\3\end{matrix}\right\}
-
\left\{\begin{matrix}3\;2\\1\end{matrix}\right\}
~,\quad
\begin{matrix}1&3\\2&\end{matrix}
=
\left\{\begin{matrix}1\;3\\2\end{matrix}\right\}
-
\left\{\begin{matrix}2\;3\\1\end{matrix}\right\}
~,\quad
\begin{matrix}2&3\\1&\end{matrix}
=
\left\{\begin{matrix}2\;3\\1\end{matrix}\right\}
-
\left\{\begin{matrix}1\;3\\2\end{matrix}\right\}
~,\quad
\]
the first two being standard polytabloids.
The \emph{Specht module} $S^\mu$ is the submodule of $M^\mu$ spanned by the polytabloids $\boldsymbol{e}_t$ for all $\mu$-tableaux $t$.
Since $g\boldsymbol{e}_t=\boldsymbol{e}_{gt}$, it is also the orbit span of any one polytabloid: $S^\mu=\operatorname{span}_{\mathbb{R}}\{g\boldsymbol{e}_t:g\in\mathfrak{S}_n\}$.

\begin{proposition}[Specht classification and Young's rule; see \cite{Sagan}*{\S\S 2.4--2.7, 2.11}]
\label{prop:young-rule}
Let $S^\mu$ denote the real Specht module obtained from the integral polytabloid lattice by extension of scalars to $\mathbb{R}$.  The modules $S^\mu$, $\mu \vdash n$, form a complete list of pairwise non-isomorphic irreducible $\mathfrak{S}_n$-modules over $\mathbb{R}$, and each remains irreducible after complexification.  Moreover, each permutation module decomposes according to \textbf{Young's rule},
$
M^\mu \cong \bigoplus_{\lambda \unrhd \mu} m_{\lambda \mu} S^\lambda,
$
where the Kostka number $m_{\lambda\mu}$ is the multiplicity of $S^\lambda$ in $M^\mu$.
\end{proposition}

\begin{proposition}[Standard polytabloid basis theorem; see \cite{Sagan}*{\S 2.5}]
\label{prop:standard-polytabloid-basis}
$\{\boldsymbol{e}_t : t \text{ is a standard $\mu$-tableau}\}$ forms a basis for $S^\mu$.
\end{proposition}

Consequently, the regular representation decomposes over $\mathbb{R}$ as
\[
\mathbb{R}[\mathfrak{S}_n]\cong M^{(1^n)}
\cong\bigoplus_{\mu\vdash n}(\dim S^\mu)S^\mu,
\]
where $mV$ denotes the direct sum of $m$ copies of $V$; the dimensions are given by the hook-length formula \cite{Sagan}*{\S 3.10}.

\zcref{prop:young-rule} also gives the decomposition of the defining representation $M^{(n-1,1)} \cong S^{(n)} \oplus S^{(n-1,1)}$, where $S^{(n-1,1)}$ is the \emph{standard representation}.
In $M^{(n-1,1)}$, the trivial representation $S^{(n)}$ is the $1$-dimensional subspace spanned by the sum of all $(n-1,1)$-tabloids, $\sum_{\ell=1}^n \{t_\ell\}$, while $S^{(n-1,1)}$ is spanned by the standard polytabloids $\boldsymbol{e}_{t_\ell}=\{t_\ell\} - \{t_1\}$, $2\leq \ell \leq n$, per \zcref{prop:standard-polytabloid-basis}.
Using this correspondence between $\{t_\ell\}$ and the coordinate vector $\boldsymbol{e}_\ell$ of $\mathbb{R}^n$, we can identify $\sum_{\ell=1}^n \{t_\ell\}$ with $\sum_{\ell=1}^n \boldsymbol{e}_\ell$, or the all-$1$ vector $\boldsymbol{1}^{(n)}$ in the coordinate basis for $\mathbb{R}^n$.
Because every row sum of $\mathcal{L}_G$ is zero, this vector is an eigenvector with the smallest eigenvalue $\lambda_1(\mathcal{L}_G)=0$.
By the same token, we can identify $\{t_\ell\}-\{t_1\}$ with $\boldsymbol{e}_\ell-\boldsymbol{e}_1$, $2\leq\ell\leq n$, which form a basis for the subspace of mean-zero vectors in $\mathbb{R}^n$ (\emph{i.e.,} the orthogonal complement to $\boldsymbol{1}^{(n)}$).
According to Rayleigh's variational principle (or the min-max theorem), $\lambda_2(\mathcal{L}_G)$ is the minimum of the Rayleigh quotient over all mean-zero vectors $\boldsymbol{w} \perp \boldsymbol{1}^{(n)}$, $\boldsymbol{w}\neq \boldsymbol{0}$.
So using the correspondence we deduce that $\lambda_2(\mathcal{L}_G)$ equals the minimum eigenvalue of the action of $L_G$ on $S^{(n-1,1)}$. This fact will be invoked frequently in the sequel.

\subsection{The branching rule for restricted and induced representations}

Sometimes we wish to restrict an irreducible representation $S^\mu$ of $\mathfrak{S}_n$ to $\mathfrak{S}_{n-1}$ (resp.\@ induce $S^\mu$ to $\mathfrak{S}_{n+1}$).
The restriction and induction decompositions are governed by the \textbf{branching rule}.
Given the Young diagram for $\mu \vdash n$, an \emph{inner corner} of $\mu$ is a box in $\mu$ whose removal leaves the Young diagram of a partition of $n-1$.
Any partition attained by such a removal is denoted $\mu-\Box$.
Conversely, an \emph{outer corner} of $\mu$ is a box not in $\mu$ whose addition produces the Young diagram of a partition of $n+1$.
Any partition attained by such an addition is denoted $\mu+\Box$.

\begin{proposition}[Branching rule; see \cite{Sagan}*{\S 2.8}]
\label{prop:branching}
If $\mu\vdash n$, then 
\[
S^\mu \mathord\downarrow_{\mathfrak{S}_{n-1}} \cong \bigoplus_{\mu' = \mu-\Box} S^{\mu'}
\quad\text{and}\quad
S^\mu \mathord\uparrow^{\mathfrak{S}_{n+1}} \cong \bigoplus_{\mu'=\mu+\Box} S^{\mu'}.
\]
\end{proposition}

\zcref{prop:branching} will be invoked in \zcref{sec:octopus} below, and plays an essential role in our proofs to follow.
In particular, when $S^{(n-2,1)}$ occurs in the restriction of $S^\mu$ to $\mathfrak{S}_{n-1}$, we denote the corresponding summand by
\begin{align}
\label{eq:branchspace}
\mathcal{B}_\mu(n):=\left(S^\mu\mathord\downarrow_{\mathfrak{S}_{n-1}}\right)_{(n-2,1)}.
\end{align}
The branching rule is multiplicity-free, so $\mathcal{B}_\mu(n)$ is the unique $\mathfrak{S}_{n-1}$-submodule of $S^\mu$ isomorphic to $S^{(n-2,1)}$.

We call a realization of $\mathcal{B}_\mu(n)$ \emph{branch-adapted} if the $\mathfrak{S}_{n-1}$-action is visible on an actual invariant subspace, rather than only on a quotient indexed by standard tableaux in which $n$ occupies a chosen removable box.
This distinction matters because the raw span of the corresponding standard polytabloids need not itself be $\mathfrak{S}_{n-1}$-invariant.
For further accounts of Specht modules, Young's rule, the branching rule, and the subspace-versus-quotient realizations used below, see \cite{Sagan}*{Chapter 2} and \cite{Fulton}*{Chapter 7}.

\subsection{Main results}

Having defined all the necessary terms, we can state our main theorem.
Recall that the interchange operator $L_G$ \eqref{eq:LG} is positive semidefinite on $\mathbb{R}[\mathfrak{S}_n] \cong \bigoplus_{\mu \vdash n} (\dim S^\mu) S^\mu$.
We are interested in the action of $L_G$ on each Specht module $S^\mu$, denoted $L_G|_{S^\mu}$.
With an appropriate choice of orthonormal basis for $S^\mu$, one can realize $L_G|_{S^\mu}$ as a symmetric, positive semidefinite operator.
Hence by the spectral theorem $L_G|_{S^\mu}$ has $\dim S^\mu$ nonnegative real eigenvalues, the minimum of which is denoted $\lambda_{\min}(L_G, S^\mu)$.

The spectral gap theorem of \cite{CLR} states that $\lambda_2(L_G)= \lambda_2(\mathcal{L}_G)$.
Equivalently: for all $\mu\vdash n$, $\mu\neq (n), (n-1,1)$, one has the non-strict inequality $\lambda_{\min}(L_G, S^\mu) \geq \lambda_{\min}(L_G, S^{(n-1,1)})$.
Our main theorem gives the necessary and sufficient condition for when this inequality saturates to equality.

\begin{theorem}[Uniqueness of the second eigenspace of $L_G$]
\label{thm:main}
Let $G=(V,E)$ be a connected $n$-vertex graph, $n\geq 3$, equipped with undirected edge weights $\{c_{ij}>0: ij\in E\}$. 
Then the following holds:
\begin{enumerate}[label=(\alph*),wide]
\item \label{item:4cycle} Suppose $G$ is the $4$-cycle equipped with uniform edge weights. Then
\[
\lambda_{\min}(L_G,S^{(2,2)})
=
\lambda_{\min}(L_G,S^{(3,1)}).
\]
The multiplicity of $\lambda_{\min}(L_G,S^{(2,2)})$ (resp.\@ $\lambda_{\min}(L_G,S^{(3,1)})$) is $1$ (resp.\@ $2$). For every $\mu\vdash4$ with $\mu\notin\{(4),(3,1),(2,2)\}$, we have
\[
\lambda_{\min}(L_G,S^\mu)
>
\lambda_{\min}(L_G,S^{(3,1)}).
\]
\item \label{item:not4cycle} If $G$ is any other connected weighted graph, then for all $\mu \vdash n$, $\mu \neq (n), (n-1,1)$, we have $\lambda_{\min}(L_G, S^\mu) > \lambda_{\min}(L_G, S^{(n-1,1)})$.
\end{enumerate}
\end{theorem}

In other words, except for the uniformly weighted $4$-cycle, the second eigenspace of $L_G$ is contained in the $S^{(n-1,1)}$-isotypic component of the regular representation, namely the direct sum of all copies of the standard representation.

We provide three consequences of \zcref{thm:main}. 

\begin{corollary}
If $G$ is the $4$-cycle with uniform weights, then the second eigenvalue of $L_G$ has multiplicity $8$.
If $G$ is any other connected $n$-vertex graph, the second eigenvalue of $L_G$ has multiplicity equal to $n-1$ times the multiplicity of the second eigenvalue of the graph Laplacian $\mathcal{L}_G$ (\emph{i.e.,} the random walk process). 
\end{corollary}
\begin{proof}
Recall
$
\mathbb{R}[\mathfrak{S}_n] \cong \bigoplus_{\mu \vdash n} (\dim S^\mu) S^\mu
$.
The general result follows from \zcref{thm:main}, \zcref{item:not4cycle}, and that $\dim S^{(n-1,1)}=n-1$.
If $G$ is the 4-cycle with uniform weights, the regular-representation decomposition is
$\mathbb{R}[\mathfrak{S}_4] \cong S^{(4)} \oplus 3 S^{(3,1)} \oplus 2 S^{(2,2)} \oplus 3 S^{(2,1,1)} \oplus S^{(1^4)}$. We then use the multiplicities stated in \zcref{thm:main}, \zcref{item:4cycle} to arrive at the total multiplicity $3\cdot 2+ 2\cdot 1=8$ of the second eigenvalue.
\end{proof}

\begin{example}[Multiplicity of the second eigenvalue in the exclusion process]
\label{ex:exclusion}
Let $1\leq k\leq \lfloor \frac{n}{2}\rfloor$, and $M^{(n-k,k)}$ be the permutation module on the $k$-subsets of $[n]$. 
The action of $L_G$ on $M^{(n-k,k)}$ is the positive Laplacian of the exclusion process in which $k$ vertices of $G$ are occupied with a particle, while the remaining $n-k$ vertices are unoccupied.
In particular, when $k=1$ we recover the random walk process on $G$.

By Young's rule $M^{(n-k,k)} \cong \bigoplus_{i=0}^k S^{(n-i,i)}$.
If $G$ is the $4$-cycle with uniform weights, then for $k=1$ the second eigenvalue has multiplicity $2$, as in \eqref{eq:A1eig}.  For $k=2$ it has multiplicity $2+1=3$ by \zcref{thm:main}, \zcref{item:4cycle}; the two raised eigenvectors and the additional $(2,2)$-sector eigenvector are given by \eqref{eq:A1eigA2} and \eqref{eq:3rdA2eig}, respectively.
If $G$ is any other connected graph, then \zcref{thm:main}, \zcref{item:not4cycle} implies that the multiplicity of the second eigenvalue of the exclusion process equals that of the second eigenvalue of the graph Laplacian $\mathcal{L}_G$.
This confirms all the observations we described in \zcref{sec:begin}.
\end{example}

\begin{example}[Multiplicity of the second eigenvalue in the colored exclusion process]
\label{ex:coloredexclusion}
Fix the integers $2\leq q\leq n$ and $\mu_1\geq \mu_2 \geq \cdots \geq \mu_q\geq 1$, subject to $\sum_{j=1}^q \mu_j=n$.
Set the partition $\mu=(\mu_1, \mu_2,\cdots, \mu_q) \vdash n$.
Then the action of $L_G$ on the permutation module $M^\mu$ is the positive Laplacian of the colored exclusion process on $G$, where $\mu_j$ of the $n$ particles are of color (or species) $j$, $1\leq j\leq q$.
By Young's rule (\zcref{prop:young-rule}), $M^\mu \cong \bigoplus_{\lambda \unrhd \mu} m_{\lambda \mu} S^\lambda$, where $m_{\lambda \mu}$ are the Kostka numbers \cite{Sagan}*{\S 2.11}.

If $q=2$ we recover the classical exclusion process of \zcref{ex:exclusion}.
If $G$ is the 4-cycle with uniform weights and $q=3$, the only pertinent permutation module is $M^{(2,1,1)} \cong S^{(4)} \oplus 2 S^{(3,1)} \oplus S^{(2,2)} \oplus S^{(2,1,1)}$, so by \zcref{thm:main}, \zcref{item:4cycle} the second eigenvalue of the colored exclusion process has multiplicity $2\cdot 2+ 1=5$.
If $n=q=4$ we recover the interchange process.
For every other connected graph $G$, $q\geq 3$, and partition $\mu=(\mu_1,\mu_2,\cdots, \mu_q) \vdash n$, \zcref{thm:main}, \zcref{item:not4cycle} implies that the second eigenvalue of the colored exclusion process has multiplicity equal to the Kostka number $m_{\lambda \mu}$ (where $\lambda=(n-1,1)$) times the multiplicity of the second eigenvalue of $\mathcal{L}_G$.
\end{example}

\section{The octopus induction scheme}
\label{sec:octopus}

As mentioned in \zcref{sec:history}, the spectral gap theorem of \cite{CLR} is proved via an induction scheme on $n$. 
In the representation theoretic language this scheme is explained in Cesi \cite{Cesi}*{Section 3.1} and in Alon, Kozma, and Puder \cite{AKP25}*{Section 3.1}.
We use the notation of Alon, Kozma, and Puder.

\begin{scheme}
Let $G$ be a connected $n$-vertex weighted graph, and suppose that the spectral-gap comparison is available for connected weighted graphs on $n-1$ vertices.  Then for every $\mu \vdash n$, $\mu \neq (n), (n-1, 1)$, one has
\ifhtml
\begin{equation}\lambda_{\min}(L_G,S^\mu)\geq\lambda_{\min}(L_{H\sqcup\{n\}},S^\mu)\label{ineq:octo1}\end{equation}
\begin{equation}\lambda_{\min}(L_{H\sqcup\{n\}},S^\mu)=\min_{\mu'=\mu-\Box}\lambda_{\min}(L_H,S^{\mu'})\label{eq:branching}\end{equation}
\begin{equation}\min_{\mu'=\mu-\Box}\lambda_{\min}(L_H,S^{\mu'})\geq\lambda_{\min}(L_H,S^{(n-2,1)})\label{ineq:octo2}\end{equation}
\begin{equation}\lambda_{\min}(L_H,S^{(n-2,1)})\geq\lambda_{\min}(L_G,S^{(n-1,1)}).\label{ineq:Schur}\end{equation}
\else
\begin{align}
\lambda_{\min}(L_G, S^\mu) &\geq
\lambda_{\min}(L_{H\sqcup \{n\}}, S^\mu) \label{ineq:octo1} \\
&= \min_{\mu' = \mu - \Box} \lambda_{\min}(L_H, S^{\mu'})  \label{eq:branching}\\
&\geq \lambda_{\min}(L_H, S^{(n-2,1)}) \label{ineq:octo2} \\
&\geq \lambda_{\min}(L_G, S^{(n-1,1)}) \label{ineq:Schur}.
\end{align}
\fi
Above one may freely designate any vertex of $G$ to be the vertex $n$, and the reduced graph $H$ of $G$ at vertex $n$ is defined accordingly.
\end{scheme}

The four lines have different logical roles in the induction.
\begin{itemize}[wide]
\item \textbf{Imported input.} The octopus inequality gives
$\Delta=L_G-L_{H\sqcup\{n\}}\geq0$ on every $S^\mu$; Rayleigh's principle then yields \eqref{ineq:octo1}.
\item \textbf{Internal identity.} Equality \eqref{eq:branching} follows because, on each $\mathfrak{S}_{n-1}$-submodule of the restricted representation, $L_{H\sqcup\{n\}}$ acts as $L_H$; the branching rule then decomposes $S^\mu\mathord\downarrow_{\mathfrak{S}_{n-1}}$ into its irreducible summands.
\item \textbf{Induction hypothesis.} Inequality \eqref{ineq:octo2} is the size-$(n-1)$ spectral-gap comparison; in the uniqueness proof, its equality cases are controlled by the strictness induction hypothesis.
\item \textbf{Internal graph-Laplacian comparison.} Finally, inequality \eqref{ineq:Schur} compares the second eigenvalue of the graph Laplacian $\mathcal{L}_G$ against the second eigenvalue of the reduced graph Laplacian $\mathcal{L}_H$. 
The reduced Laplacian is a Schur complement rather than a principal submatrix, so ordinary Cauchy interlacing is not the relevant statement.  The qualitative inequality and its equality criterion follow from the Rayleigh-quotient extension argument in \zcref{prop:eigcomp}; the explicit Schur-complement calculation is given in \zcref{sec:Schur}. 
\end{itemize}

Since $G$ is assumed to be connected, $\lambda_{\min}(L_G, S^{(n-1,1)})>0$. 
We search for nontrivial second eigenvectors of $L_G$ not in $S^{(n-1,1)}$, namely: for which $\mu\vdash n$, $\mu\neq (n), (n-1,1)$ does the equality
\[
\lambda_{\min}(L_G, S^\mu) = \lambda_{\min}(L_G, S^{(n-1,1)})
\]
hold?
This requires equality in all three inequalities \eqref{ineq:octo1}, \eqref{ineq:octo2}, and \eqref{ineq:Schur}.

\subsection{Equality in \eqref{ineq:octo1}} \label{sec:eqocto1}

This holds if and only if there exists a nonzero $w \in S^\mu \cap \ker \Delta$ such that $L_{H\sqcup\{n\}} w = \lambda_{\min}(L_G, S^\mu) w = \lambda_{\min}(L_{H\sqcup \{n\}}, S^\mu) w$.
It is a special case of the next lemma.

\begin{lemma}
\label{lem:kernel}
Let $V$ be a finite-dimensional real inner-product space, and let $A,A':V\to V$ be self-adjoint operators.  Suppose that
$\Delta:=A'-A$ is positive semidefinite.  Then
$\lambda_{\min}(A)=\lambda_{\min}(A')$ if and only if there exists
$w\in\ker\Delta\setminus\{0\}$ such that
$Aw=\lambda_{\min}(A)w$.
\end{lemma}
\begin{proof}
Suppose first that $w\neq0$, $\Delta w=0$, and
$Aw=\lambda_{\min}(A)w$.  Then
$A'w=(A+\Delta)w=\lambda_{\min}(A)w$, so
$\lambda_{\min}(A')\leq\lambda_{\min}(A)$.  Since
$A'=A+\Delta$ with $\Delta\geq0$, Rayleigh's variational principle gives the reverse inequality
$\lambda_{\min}(A')\geq\lambda_{\min}(A)$.

Conversely, assume
$\lambda_{\min}(A)=\lambda_{\min}(A')$, and choose a unit eigenvector
$w$ of $A'$ at its smallest eigenvalue; such a vector exists because $V$ is finite-dimensional.  Then
\[
\lambda_{\min}(A)
\leq \langle w,Aw\rangle
=\lambda_{\min}(A')-\langle w,\Delta w\rangle
=\lambda_{\min}(A)-\langle w,\Delta w\rangle.
\]
Hence $\langle w,\Delta w\rangle\leq0$.  Positive semidefiniteness gives the opposite inequality, so
$\langle w,\Delta w\rangle=0$.  Diagonalizing the self-adjoint positive semidefinite operator $\Delta$ shows that this equality forces $\Delta w=0$.  Therefore
$Aw=(A'-\Delta)w=\lambda_{\min}(A')w=\lambda_{\min}(A)w$.
\end{proof}

For the induction, the exact obstruction is the intersection of $\ker\Delta|_{S^\mu}$ with the minimum eigenspace of $L_{H\sqcup\{n\}}|_{S^\mu}$.  Under the induction hypothesis, this minimum eigenspace is contained in $\mathcal{B}_\mu(n)$.  Thus the branchwise condition $\ker\Delta\cap\mathcal{B}_\mu(n)=\{0\}$ is sufficient to make \eqref{ineq:octo1} strict; the required branchwise calculations are carried out in \zcref{sec:rigidity}.

\subsection{Equality in \eqref{ineq:octo2}}

This step uses the inductive strictness hypothesis at size $n-1$.
For $n\geq5$, the two relevant restrictions are
\ifhtml
\begin{equation}S^{(n-2,2)}\mathord\downarrow_{\mathfrak{S}_{n-1}}\cong S^{(n-3,2)}\oplus S^{(n-2,1)}\label{eq:branch22}\end{equation}
\begin{equation}S^{(n-2,1^2)}\mathord\downarrow_{\mathfrak{S}_{n-1}}\cong S^{(n-3,1^2)}\oplus S^{(n-2,1)}\label{eq:branch211}\end{equation}
\else
\begin{align}
S^{(n-2,2)}\mathord\downarrow_{\mathfrak{S}_{n-1}}
&\cong S^{(n-3,2)}\oplus S^{(n-2,1)},
\label{eq:branch22}\\
S^{(n-2,1^2)}\mathord\downarrow_{\mathfrak{S}_{n-1}}
&\cong S^{(n-3,1^2)}\oplus S^{(n-2,1)}.
\label{eq:branch211}
\end{align}
\fi
At $n=4$, the first formula is replaced by
\[
S^{(2,2)}\mathord\downarrow_{\mathfrak{S}_3}\cong S^{(2,1)},
\]
whereas the second formula remains valid.
Under the inductive strictness hypothesis at size $n-1$, equality in \eqref{ineq:octo2} can therefore occur only through the summand $\mathcal{B}_\mu(n)\cong S^{(n-2,1)}$, and only for $\mu=(n-2,2)$ or $\mu=(n-2,1^2)$.
Equivalently, every eigenvector of $L_{H\sqcup\{n\}}|_{S^\mu}$ at its minimum eigenvalue must lie in $\mathcal{B}_\mu(n)$.
The sole exceptional reduced graph at size $4$, namely the uniformly weighted $4$-cycle, is treated separately in \zcref{sec:n4} and in the passage from $n=4$ to $n=5$.
For all other irreducible representations $\mu\neq(n),(n-1,1),(n-2,2),(n-2,1^2)$, strict inequality holds in \eqref{ineq:octo2}.

\subsection{Equality in \eqref{ineq:Schur}} 
\label{sec:Schur}

For small graphs, equality or strictness in \eqref{ineq:Schur} can be checked directly. 
We now derive the Schur reduction and the equality criterion stated in \zcref{prop:eigcomp}. 

Let $G$ be connected, let $H$ be its reduction at vertex $n$, and recall $\lambda_2(\mathcal{L}_G)=\lambda_{\min}(L_G,S^{(n-1,1)})$.  Set
\[
\boldsymbol{c}:=(c_{1n},\ldots,c_{n-1,n})^{\mathsf T},\qquad s:=\sum_{i=1}^{n-1}c_{in}>0.
\]
With the last coordinate separated, write
\begin{align}
\label{eq:LGblock}
\mathcal{L}_G = \left[
\begin{array}{cccc|c}
\sum_{k\neq 1} {c}_{1k} & -{c}_{12}  &\cdots & -{c}_{1,n-1} & -{c}_{1n} \\
-{c}_{12} & \sum_{k\neq 2}{c}_{2k}  & \cdots & -{c}_{2,n-1} & -{c}_{2n}\\
\vdots & \vdots &  \ddots & \vdots& \vdots \\ 
-{c}_{1,n-1} & -{c}_{2,n-1} & \cdots & \sum_{k\neq n-1} {c}_{n-1,k} & -{c}_{n-1,n}\\ \hline
-{c}_{1n} & -{c}_{2n} &  \cdots & -{c}_{n-1,n} &  \sum_{k\neq n} {c}_{nk}
\end{array}\right]
=:
\left[
\begin{array}{c|c}
B & -\boldsymbol{c} \\ \hline
 -\boldsymbol{c}^{\mathsf T} & s
\end{array}
\right]
\end{align}
and compute its Schur complement with respect to the bottom-right block consisting of the $(n,n)$ entry:
\begin{align}
\label{eq:SchurLap}
\mathcal{L}_H:=B-\boldsymbol{c}s^{-1}\boldsymbol{c}^{\mathsf T}=
\begin{bmatrix}
\sum_{k\neq1}\tilde c_{1k} & -\tilde{c}_{12} & \cdots & -\tilde{c}_{1,n-1} \\
-\tilde{c}_{12} & \sum_{k\neq2}\tilde c_{2k} & \cdots & -\tilde{c}_{2,n-1} \\
\vdots & \vdots & \ddots & \vdots \\
-\tilde{c}_{1,n-1} & -\tilde{c}_{2,n-1} & \cdots & \sum_{k\neq n-1}\tilde c_{n-1,k}
\end{bmatrix}.
\end{align}
In \eqref{eq:SchurLap}, the negatives of the off-diagonal entries 
\[
\tilde{c}_{ij} = {c}_{ij} + \frac{{c}_{in} {c}_{jn}}{s}\qquad (1\leq i,j \leq n-1,~ i\neq j)
\]
are precisely the reduced edge weights from \eqref{eq:redweights}.  Moreover,
\[
(\mathcal L_H)_{ii}
=\sum_{\substack{1\leq k\leq n\\k\neq i}}c_{ik}-\frac{c_{in}^2}{s}
=\sum_{\substack{1\leq k\leq n-1\\k\neq i}}\tilde c_{ik},
\]
so every row and column of \eqref{eq:SchurLap} sums to $0$.
Thus \eqref{eq:SchurLap} is exactly the Laplacian of the reduced graph $H$, which justifies the notation $\mathcal{L}_H$.
We have analogously $\lambda_2(\mathcal{L}_H) = \lambda_{\min}(L_H, S^{(n-2,1)})$.

\begin{lemma}
\label{lem:Schur}
For every $\boldsymbol{w}\in \mathbb{R}^{n-1}$,
\[
\min_{y\in \mathbb{R}} \left(\begin{bmatrix}\boldsymbol{w}^{\mathsf T} & y\end{bmatrix} \mathcal{L}_G \begin{bmatrix} \boldsymbol{w} \\ y\end{bmatrix}\right) =\boldsymbol{w}^{\mathsf T} \mathcal{L}_H \boldsymbol{w},
\]
and the minimum is attained exactly at $y^*=\frac{\boldsymbol{c}^{\mathsf T} \boldsymbol{w}}{s}$.
\end{lemma}
\begin{proof}
In the block notation of \eqref{eq:LGblock}, the energy under the minimum reads 
\[
\begin{bmatrix}\boldsymbol{w}^{\mathsf T} & y\end{bmatrix} \left[
\begin{array}{c|c}
B & -\boldsymbol{c} \\ \hline
 -\boldsymbol{c}^{\mathsf T} & s
\end{array}
\right] \begin{bmatrix} \boldsymbol{w} \\ y\end{bmatrix}
=
s\left(y^2 - 2\frac{\boldsymbol{c}^{\mathsf T} \boldsymbol{w}}{s} y\right) +\boldsymbol{w}^{\mathsf T} B \boldsymbol{w}
=
s\left(y- \frac{\boldsymbol{c}^{\mathsf T} \boldsymbol{w}}{s}\right)^2  + \left(\boldsymbol{w}^{\mathsf T}B\boldsymbol{w} - \frac{(\boldsymbol{c}^{\mathsf T} \boldsymbol{w})^2}{s}\right).
\]
Being a convex quadratic function of $y$, this expression is minimized at $y^* = \frac{\boldsymbol{c}^{\mathsf T}\boldsymbol{w}}{s}$, returning the value $\boldsymbol{w}^{\mathsf T}B \boldsymbol{w} - \boldsymbol{w}^{\mathsf T} \boldsymbol{c} s^{-1} \boldsymbol{c}^{\mathsf T} \boldsymbol{w} = \boldsymbol{w}^{\mathsf T}\mathcal{L}_H \boldsymbol{w}$, per \eqref{eq:SchurLap}.
\end{proof}

\begin{proposition}[Second Laplacian eigenvalues under Schur reduction]
\label{prop:eigcomp}
For connected $G$, one has $\lambda_2(\mathcal{L}_H)\geq\lambda_2(\mathcal{L}_G)$, which is \eqref{ineq:Schur}.  Equality holds if and only if $\mathcal{L}_G$ has a second eigenvector that vanishes at the removed vertex.  Equivalently, there is a nonzero $\boldsymbol{w}\in\mathbb R^{n-1}$ such that
$
\begin{bmatrix}\boldsymbol{w}\\0\end{bmatrix}
$
is a second eigenvector of $\mathcal L_G$.  Any such vector automatically satisfies $\boldsymbol{w}\perp\boldsymbol{1}^{(n-1)}$,
\begin{align}
\label{eq:cw}
\sum_{i=1}^{n-1}c_{in}\boldsymbol{w}(i)=0,
\end{align}
and $\mathcal L_H\boldsymbol{w}=\lambda_2(\mathcal L_H)\boldsymbol{w}$.
\end{proposition}

\begin{proof}
According to Rayleigh's variational principle,
\[
\lambda_2(\mathcal{L}_H) = \min_{\substack{\boldsymbol{w}\in \mathbb{R}^{n-1} \\ \|\boldsymbol{w}\|=1,~ \boldsymbol{w}\perp \boldsymbol{1}^{(n-1)}}} \boldsymbol{w}^{\mathsf T} \mathcal{L}_H \boldsymbol{w}
\quad
\text{and}
\quad
\lambda_2(\mathcal{L}_G) = \min_{\substack{\boldsymbol{u}\in \mathbb{R}^{n} \\ \|\boldsymbol{u}\|=1,~ \boldsymbol{u}\perp \boldsymbol{1}^{(n)}}} \boldsymbol{u}^{\mathsf T} \mathcal{L}_G \boldsymbol{u}.
\]
By \zcref{lem:Schur}, 
\begin{align}
\label{eq:l2LH}
\lambda_2(\mathcal{L}_H) = \min_{\substack{\boldsymbol{w}\in \mathbb{R}^{n-1} \\ \|\boldsymbol{w}\|=1, ~\boldsymbol{w}\perp \boldsymbol{1}^{(n-1)}}} \begin{bmatrix} \boldsymbol{w}^{\mathsf T} &  y^*\end{bmatrix}  \mathcal{L}_G \begin{bmatrix} \boldsymbol{w} \\ y^* \end{bmatrix},
\end{align}
where $y^* = y^*(\boldsymbol{w}) := \frac{\boldsymbol{c}^{\mathsf T} \boldsymbol{w}}{s}$.
This resembles the variational form of $\lambda_2(\mathcal{L}_G)$, except that in general $\left[\begin{matrix} \boldsymbol{w} \\ y^* \end{matrix}\right]$ is neither orthogonal to $\boldsymbol{1}^{(n)}$ nor a unit vector in $\mathbb{R}^n$.

To make the comparison apt, we subtract from $\left[\begin{matrix} \boldsymbol{w}\\ y^*\end{matrix}\right]$ its orthogonal projection onto $\boldsymbol{1}^{(n)}$,
\[
\frac{1}{n}\left(\sum_{i=1}^{n-1} \boldsymbol{w}(i)+ y^*\right) \boldsymbol{1}^{(n)} = \frac{y^*}{n} \boldsymbol{1}^{(n)},
\]
where we used $\boldsymbol{w}\perp \boldsymbol{1}^{(n-1)}$. 
Let us adopt the shorthand 
\[
\boldsymbol{u}^* = \boldsymbol{u}^*(\boldsymbol{w}):=\begin{bmatrix} \boldsymbol{w}\\ y^*\end{bmatrix} - \frac{y^*}{n}\boldsymbol{1}^{(n)},
\]
and note that $\boldsymbol{u}^* \perp \boldsymbol{1}^{(n)}$.
Since $\mathcal{L}_G \boldsymbol{1}^{(n)}=0$, we have $\left[\begin{matrix} \boldsymbol{w}^{\mathsf T} &  y^*\end{matrix} \right] \mathcal{L}_G\left[\begin{matrix} \boldsymbol{w} \\ y^* \end{matrix} \right]= (\boldsymbol{u}^*)^{\mathsf T} \mathcal{L}_G \boldsymbol{u}^*$.
Meanwhile,
\begin{align*}
\|\boldsymbol{u}^*\|^2 &= \begin{bmatrix} \boldsymbol{w}^{\mathsf T} & y^*\end{bmatrix} \begin{bmatrix} \boldsymbol{w} \\ y^*\end{bmatrix} - 2 \frac{y^*}{n}  \begin{bmatrix} \boldsymbol{w}^{\mathsf T} & y^*\end{bmatrix} \boldsymbol{1}^{(n)} + \frac{(y^*)^2}{n}  = \|\boldsymbol{w}\|^2+ (y^*)^2 - \frac{(y^*)^2}{n} = 1 + \left(1-\frac{1}{n}\right) (y^*)^2. 
\end{align*}
Therefore \eqref{eq:l2LH} is equal to
\begin{equation}
\label{ineq:Rayleigh}
\begin{aligned}
\min_{\substack{\boldsymbol{w}\in \mathbb{R}^{n-1} \\ \|\boldsymbol{w}\|=1, ~\boldsymbol{w}\perp \boldsymbol{1}^{(n-1)}}} \frac{(\boldsymbol{u}^*)^{\mathsf T} \mathcal{L}_G \boldsymbol{u}^*}{\|\boldsymbol{u}^*\|^2} \left(1+\left(1-\frac{1}{n}\right)(y^*)^2\right)
&\geq
\min_{\substack{\boldsymbol{w}\in \mathbb{R}^{n-1} \\ \|\boldsymbol{w}\|=1,~ \boldsymbol{w}\perp \boldsymbol{1}^{(n-1)}}} \frac{(\boldsymbol{u}^*)^{\mathsf T} \mathcal{L}_G \boldsymbol{u}^*}{\|\boldsymbol{u}^*\|^2} \\
&\geq 
\min_{\substack{\boldsymbol{u}\in \mathbb{R}^{n} \\ \|\boldsymbol{u}\|=1,~ \boldsymbol{u}\perp \boldsymbol{1}^{(n)}}} \boldsymbol{u}^{\mathsf T} \mathcal{L}_G \boldsymbol{u} = \lambda_2(\mathcal{L}_G),
\end{aligned}
\end{equation}
which proves the first statement of the Proposition.

Now suppose $\lambda_2(\mathcal{L}_H ) =\lambda_2(\mathcal{L}_G)$, and let $\boldsymbol{w}$ attain the minimum in \eqref{eq:l2LH}.  For this $\boldsymbol{w}$, the two factors in the first line of \eqref{ineq:Rayleigh} satisfy
\[
\frac{(\boldsymbol{u}^*)^{\mathsf T}\mathcal{L}_G\boldsymbol{u}^*}{\|\boldsymbol{u}^*\|^2}\geq\lambda_2(\mathcal{L}_G)
\qquad\text{and}\qquad
1+\left(1-\frac1n\right)(y^*)^2\geq1.
\]
Their product equals $\lambda_2(\mathcal{L}_G)$, so both inequalities are equalities.  Hence $y^*=0$, which is exactly \eqref{eq:cw}, and $\boldsymbol{u}^*=\left[\begin{matrix}\boldsymbol{w}\\0\end{matrix}\right]$ attains the Rayleigh minimum for $\mathcal{L}_G$.  Thus $\left[\begin{matrix}\boldsymbol{w}\\0\end{matrix}\right]$ is a second eigenvector of $\mathcal{L}_G$.
Conversely, suppose that $\left[\begin{matrix}\boldsymbol{w}\\0\end{matrix}\right]$ is a nonzero second eigenvector of $\mathcal{L}_G$.
Because its eigenvalue is positive and $\ker\mathcal{L}_G=\mathbb{R}\boldsymbol{1}^{(n)}$, it is orthogonal to $\boldsymbol{1}^{(n)}$, and hence $\boldsymbol{w}\perp\boldsymbol{1}^{(n-1)}$.
The last coordinate of the eigenvalue equation is $-\boldsymbol{c}^{\mathsf T}\boldsymbol{w}=0$, which is exactly \eqref{eq:cw}.
After normalizing $\boldsymbol{w}$ and inserting it into \eqref{eq:l2LH}, we obtain $\lambda_2(\mathcal{L}_H)\leq\lambda_2(\mathcal{L}_G)$.  Together with the already proved reverse inequality, this yields equality.
Moreover, every such $\boldsymbol{w}$ is automatically a second eigenvector of $\mathcal{L}_H$ because of
\begin{align*}
\mathcal{L}_G \begin{bmatrix} \boldsymbol{w}\\ 0\end{bmatrix} &= \begin{bmatrix} B & -\boldsymbol{c} \\ -\boldsymbol{c}^{\mathsf T} & s\end{bmatrix} \begin{bmatrix} \boldsymbol{w}\\ 0\end{bmatrix} = \begin{bmatrix} B\boldsymbol{w} \\ -\boldsymbol{c}^{\mathsf T} \boldsymbol{w} \end{bmatrix} = \begin{bmatrix} B\boldsymbol{w}\\ 0\end{bmatrix} = \lambda_2(\mathcal{L}_G)\begin{bmatrix} \boldsymbol{w}\\ 0\end{bmatrix},\\
\mathcal{L}_H \boldsymbol{w} &= (B-(-\boldsymbol{c})s^{-1} (-\boldsymbol{c}^{\mathsf T}))\boldsymbol{w} = B\boldsymbol{w} - 0\boldsymbol{w} = \lambda_2(\mathcal{L}_G)\boldsymbol{w},
\end{align*}
and the inequality $\lambda_2(\mathcal{L}_H) \geq \lambda_2(\mathcal{L}_G)$ proved above.
This completes the proof.
\end{proof}

\begin{remark}[Why the extension hypothesis is essential]
\label{rem:warning}
The requirement in \zcref{prop:eigcomp} that $\left[\begin{matrix}\boldsymbol{w}\\0\end{matrix}\right]$ be a second eigenvector of $\mathcal L_G$ cannot be omitted.
If $\boldsymbol{w}\in \mathbb{R}^{n-1}$ is a second eigenvector of $\mathcal{L}_H$, and $\sum_{i=1}^{n-1} c_{in} \boldsymbol{w}(i)=0$, then $\left[\begin{matrix} \boldsymbol{w}\\0\end{matrix}\right] \in \mathbb{R}^n$ may or may not be an eigenvector of $\mathcal{L}_G$.
Even when the extended vector is an eigenvector of $\mathcal L_G$, its eigenvalue need not be the second eigenvalue. 
We will encounter this issue in the proofs of \zcref{thm:existunique} and \zcref{prop:5cycle}, respectively.
For a concrete example see also \zcref{rem:subtlewarning}.
\end{remark}

\section{Proof of the main theorem}
\label{sec:mainproof}

We proceed in four stages: $n=3$; $n=4$; $n=5$; and $n\geq 6$. 
The induction process uses the octopus scheme of \zcref{sec:octopus}.

\subsection{The case $n=3$}
\label{sec:n3}

We first verify \zcref{thm:main} for $n=3$.
Given that the irreducible representations of $\mathfrak{S}_3$ are $(3), (2,1), (1^3)$, it suffices to check that $\lambda_{\min}(L_G, S^{(1^3)}) > \lambda_{\min}(L_G, S^{(2,1)})$. Indeed, $S^{(1^3)}$ is the 1-dimensional \emph{sign representation} of $\mathfrak{S}_3$, and we denote its basis vector by $\boldsymbol{e}$. 
For every transposition $(i,j)$, $1\leq i<j\leq 3$, we have $(i,j)\boldsymbol{e} = -\boldsymbol{e}$, and hence $L_G \boldsymbol{e} = \sum_{1\leq i<j\leq 3} {c}_{ij} (\operatorname{Id}- (i,j))\boldsymbol{e} = 2 ({c}_{12}+ {c}_{13} + {c}_{23})\boldsymbol{e}$.
Thus $\lambda_{\min}(L_G, S^{(1^3)}) = 2({c}_{12}+ {c}_{13} + {c}_{23})$.
A direct characteristic-polynomial calculation for the $3\times3$ graph Laplacian gives 
\[
\lambda_{\min}(L_G, S^{(2,1)}) = \lambda_2(\mathcal{L}_G)= ({c}_{12} + {c}_{13} +{c}_{23}) - \sqrt{\frac{1}{2}[({c}_{12} - {c}_{13})^2 +({c}_{13} - {c}_{23})^2
+
({c}_{23} - {c}_{12})^2 ]}.
\]
Therefore $\lambda_{\min}(L_G, S^{(1^3)}) > \lambda_{\min}(L_G, S^{(2,1)}) $ so long as one of ${c}_{12}, {c}_{23}, {c}_{13}$ is positive.

This result can be readily extended to all $n\geq 3$.

\begin{proposition}[Interchange on the sign representation]
\label{prop:signrep}
The interchange operator $L_G$ acts as the scalar $2 \sum_{1\leq i<j\leq n} {c}_{ij}$ on $S^{(1^n)}$. If $n\geq 3$, we have $2 \sum_{1\leq i<j\leq n} {c}_{ij} >  \lambda_2(\mathcal{L}_G) = \lambda_{\min}(L_G, S^{(n-1,1)}) $ whenever one of the edge weights ${c}_{ij}$ is positive.
\end{proposition}
\begin{proof}
Again $(i,j)\boldsymbol{e}=-\boldsymbol{e}$ for every transposition acting on the single basis vector $\boldsymbol{e}$ of $S^{(1^n)}$, so the first statement follows. 
For the second statement, fix $i\in[n]$ and project the coordinate vector $\boldsymbol e_i$ onto $(\boldsymbol 1^{(n)})^\perp$.  Its Rayleigh quotient is
\[
\frac{n}{n-1}\sum_{j\neq i}c_{ij},
\]
so the variational principle gives
\[
\lambda_2(\mathcal L_G)\leq\frac{n}{n-1}\sum_{j\neq i}c_{ij}
\leq\frac32\sum_{j\neq i}c_{ij}
<2\sum_{1\leq k<\ell\leq n}c_{k\ell}.
\]
The final inequality is strict whenever at least one edge weight is positive.
\end{proof}

\subsection{The case $n=4$}
\label{sec:n4}

There are three irreducible Specht sectors in which an additional second eigenvector could occur: $(2,2)$, $(2,1^2)$, and $(1^4)$.
The sign sector $(1^4)$ is excluded by \zcref{prop:signrep}, leaving the two shapes $(2,2)$ and $(2,1^2)$.
They play different roles.  We begin with $(2,2)$ because it is the only one that can support equality, and because the explicit two-dimensional calculation in this sector reveals the exceptional uniformly weighted $4$-cycle.  After completing that calculation, we return to $(2,1^2)$ and explain why its corresponding branch is always rigid.

By the octopus induction scheme of \zcref{sec:octopus}, equality at the interchange gap requires equality in \eqref{ineq:octo1}, \eqref{ineq:octo2}, and \eqref{ineq:Schur}.
For either remaining shape, the restrictions $\mu'=\mu-\Box$ include $(2,1)$, so equality in \eqref{ineq:octo2} is possible.
The restriction is multiplicity-free; whenever the minimum in \eqref{eq:branching} is attained only at $S^{(n-2,1)}$, the corresponding minimum eigenspace of $L_{H\sqcup\{n\}}|_{S^\mu}$ lies in the selected branch $\mathcal{B}_\mu(n)$.
Thus the first substantive question is whether the octopus operator has a kernel vector in that branch.

The next theorem gives the complete analysis in the $(2,2)$ sector.

\begin{theorem}
\label{thm:existunique}
Suppose $G$ is a connected $4$-vertex graph. Then 
$\lambda_{\min}(L_G, S^{(2,2)}) = \lambda_{\min}(L_G, S^{(3,1)})$ if and only if $G$ is the $4$-cycle with uniform weights, in which case $\lambda_{\min}(L_G, S^{(2,2)})$ (resp.\@ $\lambda_{\min}(L_G, S^{(3,1)})$) has multiplicity $1$ (resp.\@ $2$).
Otherwise, $\lambda_{\min}(L_G, S^{(2,2)}) > \lambda_{\min}(L_G, S^{(3,1)})$.
\end{theorem}

To prove \zcref{thm:existunique} we perform an explicit computation on the polytabloids. 
There are two standard $(2,2)$-tableaux,
\[
t_2:=\begin{smallmatrix}   
\boxed{1}&\boxed{3}\\[-1pt]\boxed{2}&\boxed{4}\end{smallmatrix}
\quad\text{and}\quad
t_3:=\begin{smallmatrix}\boxed{ 1}&\boxed{ 2}\\[-1pt]\boxed{ 3}&\boxed{ 4}\end{smallmatrix}
~,
\] 
and $\{\boldsymbol{e}_{t_2}, \boldsymbol{e}_{t_3}\}$ forms a basis for $S^{(2,2)}$.
The following lemma is the crux of the computation.

\begin{lemma}
\label{lem:four-vertex-two-row-kernel}
Suppose $G$ is a connected $4$-vertex graph, and let $H$ be the reduced graph of $G$ at vertex $4$. 
\begin{enumerate}[label=(\arabic*), wide]
\item \label{item:4deg3} If vertex $4$ has degree $3$ in $G$, \emph{i.e.,} $c_{14}, c_{24}, c_{34}>0$, then $\ker\Delta|_{S^{(2,2)}}=\{0\}$. 
\item \label{item:4deg2}  Otherwise, after relabeling vertices $1,2,3$, assume that $c_{34}=0$ and at least one of $c_{14}$ and $c_{24}$ is positive.
Then the linear combination $\gamma_2\boldsymbol{e}_{t_2}+\gamma_3\boldsymbol{e}_{t_3}\in \ker \Delta |_{S^{(2,2)}}$ if and only if $(c_{14}-c_{24})\gamma_2=(c_{14}+2{c_{24}})\gamma_3$.
\end{enumerate}
\end{lemma}

\begin{proof}
Set $\mu=(2,2)$.
Let $\mathcal E=(\boldsymbol{e}_{t_2},\boldsymbol{e}_{t_3})$.  Using the standard action of transpositions on polytabloids and the straightening algorithm, we obtain the matrix $[\Delta|_{S^\mu}]_{\mathcal E}$.
Again we abbreviate $c_{k4}$ to $c_k$ for $k\in \{1,2,3\}$ and set $s= \sum_{k=1}^3 c_k$.
A direct expansion yields
\[
[\Delta|_{S^\mu}]_{\mathcal E}= s^{-1}
\begin{bmatrix}
c_1^2+3c_1 c_3 - c_1 c_2 + 2c_3^2 + c_2 c_3 & -c_1^2-2c_1 c_2 + 2c_2 c_3 + c_3^2 \\
-c_1^2-2c_1 c_3 + 2c_2 c_3 + c_2^2 & c_1^2+3c_1 c_2 - c_1 c_3 + 2c_2^2 + c_2 c_3
\end{bmatrix}.
\]
One verifies that $\operatorname{tr}[\Delta|_{S^\mu}]_{\mathcal E} = 2s^{-1}(c_1^2+c_2^2+c_3^2+c_1c_2 + c_2 c_3 + c_3 c_1)$ and $\det [\Delta|_{S^\mu}]_{\mathcal E} = 12 s^{-1} c_1 c_2 c_3$.
On the one hand, if $c_1, c_2, c_3>0$, then $\det [\Delta|_{S^\mu}]_{\mathcal E} >0$, which implies that $\ker \Delta|_{S^\mu} = \{0\}$, proving \zcref{item:4deg3}.
On the other hand, suppose $c_3=0$ and at least one of $c_1$ and $c_2$ is positive. Then
\[
[\Delta|_{S^\mu}]_{\mathcal E} = s^{-1} \begin{bmatrix}
c_1^2 - c_1 c_2 & -c_1^2 -2 c_1 c_2 \\ -c_1^2 +c_2^2 & c_1^2 + 3c_1 c_2 + 2c_2^2\end{bmatrix}
=
s^{-1} \begin{bmatrix} c_1  \\ -(c_1+c_2)\end{bmatrix}
\begin{bmatrix} c_1 - c_2 &  -(c_1+2c_2) \end{bmatrix},
\]
a rank-$1$ matrix.
The displayed kernel vector follows by solving $[\Delta|_{S^\mu}]_{\mathcal E}(\gamma_2,\gamma_3)^{\mathsf T}=0$.
\end{proof}

\begin{remark}
The later branchwise theory of \zcref{sec:rigidity} subsumes Lemma~\ref{lem:four-vertex-two-row-kernel}, but we give the direct polytabloid calculation here because this was the route by which the uniformly weighted $4$-cycle was rigorously discovered.  
In the smallest nonlinear Specht module, the rank drop at degree at most $2$ and the surviving kernel direction can be read directly from the displayed $2\times2$ matrix.  The branch-adapted calculation in \zcref{sec:rigidity} explains why this phenomenon persists only in the two-row branch, and why degree at least $3$ restores rigidity.
\end{remark}

\begin{proof}[Proof of \zcref{thm:existunique}]
We now apply the octopus induction scheme from \zcref{sec:octopus}.
Suppose $G$ is a connected $4$-vertex graph whose maximum degree is $3$.
After relabeling the vertices, assume that vertex $4$ has degree $3$.
Lemma~\ref{lem:four-vertex-two-row-kernel}, part~\ref{item:4deg3}, together with Lemma~\ref{lem:kernel}, makes \eqref{ineq:octo1} strict.

It remains to consider connected $4$-vertex graphs of maximum degree $2$, namely a $4$-cycle or a $4$-path.  After relabeling, assume that the edge weights satisfy $c_{14}, c_{24}, c_{13} >0$, $c_{23}\geq 0$, and $c_{12}=c_{34}=0$.
Equality $\lambda_{\min}(L_G,S^{(2,2)})=\lambda_{\min}(L_G,S^{(3,1)})$ requires the following three conditions simultaneously:

\begin{itemize}[wide]
\item Equality is attained in \eqref{ineq:octo1}: By Lemma~\ref{lem:four-vertex-two-row-kernel}, \zcref{item:4deg2}, this holds if and only if $(c_{14}+2c_{24})\boldsymbol{e}_{t_2} + (c_{14}-c_{24}) \boldsymbol{e}_{t_3}$ is an eigenvector of $L_{H\sqcup \{4\}}$ with eigenvalue $\lambda_{\min}(L_{H\sqcup \{4\}}, S^{(2,2)})$.
\item Equality is attained in \eqref{ineq:octo2}: Since $S^{(2,2)}\mathord\downarrow_{\mathfrak{S}_3}\cong S^{(2,1)}$, the minimum in the branching relation is exactly $\lambda_{\min}(L_H,S^{(2,1)})$.
\item Equality is attained in \eqref{ineq:Schur}: By \zcref{prop:eigcomp}, this holds if and only if the graph Laplacian $\mathcal{L}_G$ has a second eigenvector of the form $\left[\begin{matrix} \boldsymbol{w} \\ 0\end{matrix}\right]$ for some nonzero $\boldsymbol{w}\in \mathbb{R}^3$ which satisfies $\boldsymbol{w}\perp \boldsymbol{1}^{(3)}$ and $\sum_{i=1}^3 c_{i4}\boldsymbol{w}(i)=0$. 
In this case, $\boldsymbol{w}$ is automatically a second eigenvector of the reduced graph Laplacian $\mathcal{L}_H$.
\end{itemize}

The key is to bridge the first and third conditions.
Since $S^{(2,2)}\mathord\downarrow_{\mathfrak{S}_3}\cong S^{(2,1)}$, the action of $L_{H\sqcup\{4\}}$ on $S^{(2,2)}$ is isomorphic to the action of $L_H$ on $S^{(2,1)}$.
Under the standard identification of $S^{(2,1)}$ with the mean-zero subspace of $\mathbb{R}^3$, the coefficient vector $\gamma_2 \boldsymbol{e}_{t_2}+\gamma_3 \boldsymbol{e}_{t_3}$ corresponds, up to a nonzero overall scalar, to
$\gamma_2(\boldsymbol{e}_2-\boldsymbol{e}_1)+\gamma_3(\boldsymbol{e}_3-\boldsymbol{e}_1)$.
Thus $\gamma_2\boldsymbol e_{t_2}+\gamma_3\boldsymbol e_{t_3}$ is a minimum eigenvector of $L_{H\sqcup\{4\}}$ if and only if $\gamma_2(\boldsymbol e_2-\boldsymbol e_1)+\gamma_3(\boldsymbol e_3-\boldsymbol e_1)$ is a second eigenvector of $\mathcal L_H$.

We therefore examine the spectrum of
\[
\mathcal{L}_H = \begin{bmatrix}
\tilde{c}_{12}+\tilde{c}_{13} & -\tilde{c}_{12} & -\tilde{c}_{13}\\
-\tilde{c}_{12} & \tilde{c}_{12}+\tilde{c}_{23} & -\tilde{c}_{23}\\
-\tilde{c}_{13} & -\tilde{c}_{23} & \tilde{c}_{13}+\tilde{c}_{23} 
\end{bmatrix}
\]
in the standard basis $\{\boldsymbol{e}_1, \boldsymbol{e}_2, \boldsymbol{e}_3\}$.
Besides a simple eigenvalue $0$, the other two eigenvalues of $\mathcal{L}_H$ are
$\lambda_{\pm} =(\tilde{c}_{12} + \tilde{c}_{23} + \tilde{c}_{13}) \pm \sqrt{\frac{1}{2} \left[ (\tilde{c}_{12}- \tilde{c}_{23})^2 +(\tilde{c}_{23} - \tilde{c}_{13})^2 + (\tilde{c}_{13} - \tilde{c}_{12})^2\right]}$.

We now prove the if and only if characterization stated in the theorem.
There are two alternatives to consider: $\lambda_+ >\lambda_-$, or $\lambda_+ = \lambda_-$.

\underline{The case $\lambda_+ > \lambda_-$:} Then the second eigenspace of $\mathcal{L}_H$ is $1$-dimensional. 
We wish to show that there exists a line of coefficient vectors $\left[\begin{matrix} \gamma_2 \\ \gamma_3\end{matrix}\right] \in \mathbb{R}^2$ such that the following four items hold simultaneously:
\begin{enumerate}[label=(\roman*), wide]
\item \label{item:cri3}  $\left[\begin{matrix} \gamma_2 \\ \gamma_3\end{matrix}\right]$ is a scalar multiple of $\left[\begin{matrix} {c}_{14} +2{c}_{24} \\ {c}_{14} - {c}_{24}\end{matrix}\right]$, by Lemma~\ref{lem:four-vertex-two-row-kernel}, \zcref{item:4deg2};
\item \label{item:extend} $\boldsymbol{w}=(-\gamma_2-\gamma_3)\boldsymbol{e}_1+\gamma_2\boldsymbol{e}_2+\gamma_3\boldsymbol{e}_3$, and $\left[\begin{matrix}\boldsymbol{w}\\0\end{matrix}\right]$ is a second eigenvector of $\mathcal L_G$ with eigenvalue $\lambda_-$, by \zcref{prop:eigcomp};
\item \label{item:cri2} $\sum_{i=1}^3 c_{i4} \boldsymbol{w}(i) = c_{14}(-\gamma_2-\gamma_3) + c_{24} \gamma_2=0$, or equivalently, $\left[\begin{matrix} \gamma_2 \\ \gamma_3\end{matrix}\right]$ is a scalar multiple of $\left[\begin{matrix} {c}_{14} \\ {c}_{24} - {c}_{14} \\ \end{matrix}\right]$,  by \zcref{prop:eigcomp}; 
\item \label{item:cri1} $\boldsymbol{w}=(-\gamma_2-\gamma_3)\boldsymbol{e}_1 + \gamma_2 \boldsymbol{e}_2 + \gamma_3\boldsymbol{e}_3$ is, up to scalar multiples, the unique second eigenvector of $\mathcal{L}_H$ (with eigenvalue $\lambda_-$), by \zcref{prop:eigcomp}.
\end{enumerate}
\zcref{item:cri3} implies equality in \eqref{ineq:octo1}, and \zcref{item:extend} through \zcref{item:cri1} implies equality in \eqref{ineq:Schur}.

Observe that \zcref{item:cri3} and  \zcref{item:cri2} hold simultaneously if and only if ${c}_{14}={c}_{24}$, in which case $\left[\begin{matrix} \gamma_2 \\ \gamma_3\end{matrix}\right]$ is a scalar multiple of $\left[\begin{matrix} 1\\0\end{matrix}\right]$. 
Then consider \zcref{item:cri1}, where we find
\[
\mathcal{L}_H \begin{bmatrix} -1\\1\\0\end{bmatrix} = \begin{bmatrix} \tilde{c}_{12}+\tilde{c}_{13} & -\tilde{c}_{12} & -\tilde{c}_{13} \\
-\tilde{c}_{12} & \tilde{c}_{12}+ \tilde{c}_{23} & -\tilde{c}_{23}\\
-\tilde{c}_{13} & -\tilde{c}_{23} & \tilde{c}_{13} + \tilde{c}_{23}
\end{bmatrix} \begin{bmatrix} -1\\1\\0\end{bmatrix}
=
\begin{bmatrix}
-2\tilde{c}_{12} -\tilde{c}_{13} 
\\
2\tilde{c}_{12} + \tilde{c}_{23}
\\
\tilde{c}_{13} - \tilde{c}_{23}
\end{bmatrix}.
\]
In order for $\left[\begin{matrix} -1\\ 1 \\ 0\end{matrix}\right]$ to be an eigenvector of $\mathcal{L}_H$, we must have $\tilde{c}_{13} = \tilde{c}_{23}$, and the corresponding eigenvalue is $2\tilde{c}_{12} + \tilde{c}_{13}$.

Since ${c}_{12}={c}_{34}=0$, we have $\tilde{c}_{13} ={c}_{13}$, $\tilde{c}_{23} = {c}_{23}$, and $\tilde{c}_{12} = \frac{{c}_{14}{c}_{24}}{{c}_{14} + {c}_{24}}$. 
Thus the edge weights of $G$ are now determined by two parameters: ${c}_{13} = {c}_{23} =:\alpha$ and ${c}_{14}= {c}_{24} =:\beta$.

Finally consider \zcref{item:extend}. We can directly verify that $\left[-1,~1,~0,~0\right]^{\mathsf T}$ is an eigenvector of $\mathcal{L}_G$: 
\[
\mathcal{L}_G \begin{bmatrix} -1\\1\\0\\0\end{bmatrix} = \begin{bmatrix} \alpha+\beta & 0 & -\alpha & -\beta \\ 0 & \alpha+\beta & -\alpha & -\beta \\ -\alpha & -\alpha & 2\alpha & 0 \\ -\beta & -\beta & 0 & 2\beta \end{bmatrix}\begin{bmatrix} -1\\1\\0\\0\end{bmatrix} =(\alpha+\beta) \begin{bmatrix} -1 \\ 1 \\ 0 \\ 0\end{bmatrix}.
\]
To see if it is a second eigenvector of $\mathcal{L}_G$, we find the four eigenvalues of $\mathcal{L}_G$, which are
\[
0, \quad\alpha+\beta,\quad \frac{3(\alpha+\beta) \pm \sqrt{9\alpha^2-14\alpha \beta+9\beta^2}}{2}.
\]
Writing
\[
r_-:=\frac{3(\alpha+\beta)-\sqrt{9\alpha^2-14\alpha\beta+9\beta^2}}{2},
\]
we have
\[
(\alpha+\beta)-r_-
=\frac{\sqrt{9\alpha^2-14\alpha\beta+9\beta^2}-(\alpha+\beta)}{2}\geq 0,
\]
because the difference of the squares of the two nonnegative quantities in the numerator is $8(\alpha-\beta)^2$.
Moreover, equality holds if and only if $\alpha=\beta$.
Hence $\lambda_2(\mathcal{L}_G)=\alpha+\beta$ if and only if $\alpha=\beta$.
Thus the only graphs $G$ which satisfy \zcref{item:cri3} through \zcref{item:cri1} are the $4$-cycles with uniform weights $c_{13}=c_{23}=c_{14}=c_{24}>0$ (and $c_{12}=c_{34}=0$).

Since a common rescaling of all edge weights rescales the operator without changing eigenspace dimensions, it suffices to verify the multiplicities on the $4$-cycle with simple weights. 
In \zcref{ex:4cycle} we identified the two second eigenvectors of $\mathcal{L}_G$ (equivalently, the eigenvectors of $L_G|_{S^{(3,1)}}$ with the minimum eigenvalue $\lambda_{\min}(L_G, S^{(3,1)})=2$).
Meanwhile, in the same ordered basis $\mathcal E=(\boldsymbol e_{t_2},\boldsymbol e_{t_3})$, direct calculation gives
\[
[L_G|_{S^{(2,2)}}]_{\mathcal E}= \begin{bmatrix} 2 & 2\\ 0 & 6\end{bmatrix}
\]
in the basis $\{\boldsymbol{e}_{t_2}, \boldsymbol{e}_{t_3}\}$.
Thus $\lambda_{\min}(L_G, S^{(2,2)}) = 2$ with corresponding eigenvector 
\[
\boldsymbol{e}_{t_2} =
\begin{matrix} 1&3\\[-1pt] 2&4\end{matrix}
~-~
\begin{matrix} 2& 3\\[-1pt] 1& 4\end{matrix}
~-~
\begin{matrix} 1& 4\\[-1pt] 2& 3\end{matrix}
~+~
\begin{matrix}2& 4\\[-1pt] 1& 3\end{matrix}
~.
\]
Upon identifying each tabloid by the $2$-subset appearing in row $2$, we see that $\boldsymbol{e}_{t_2}$ matches the vector $[0,1,-1,-1,1,0]^{\mathsf T}$ of \eqref{eq:3rdA2eig}.

\underline{The case $\lambda_+ =\lambda_-$:} This implies the equality $\tilde{c}_{12} = \tilde{c}_{23} = \tilde{c}_{13}$, namely: $H$ is the complete graph $K_3$ with uniform weights. 
We claim that in this setting inequality \eqref{ineq:Schur} is strict, \emph{i.e.,} $\lambda_2(\mathcal{L}_H) > \lambda_2(\mathcal{L}_G)$.

As in the previous case, we have $\tilde{c}_{13} ={c}_{13}$, $\tilde{c}_{23} = {c}_{23}$, and $\tilde{c}_{12} = \frac{{c}_{14}{c}_{24}}{{c}_{14} + {c}_{24}}$. The equality of the tilded weights thus reads ${c}_{13}={c}_{23} = \frac{{c}_{14}{c}_{24}}{{c}_{14} + {c}_{24}}$.
Without loss of generality set ${c}_{13}={c}_{23}=1$, $b={c}_{14}$, and $d={c}_{24}$, with $1=\frac{bd}{b+d}$.
Since $b,d>0$, this last equation implies $b=\frac{d}{d-1}$ and $d>1$.
Note that $H$ is now the complete graph $K_3$ with simple weights, so $\lambda_2(\mathcal{L}_H)=3$.

A direct determinant calculation gives the characteristic polynomial
$\det(\lambda I-\mathcal{L}_G)=
\lambda(\lambda-3) Q(\lambda)$, where $$Q(\lambda)=\lambda^{2}+\frac{-2d^{2}-d+1}{d-1}\lambda+\frac{4d^{2}}{d-1}.$$ The eigenvalues of $\mathcal{L}_G$ are $0$, $3$, and the two roots of the quadratic polynomial $Q$. We compute $Q(0)=\frac{4d^2}{d-1}$ and $Q(3)=\frac{-2(d^{2}-3d+3)}{d-1}$. Note that $d^{2}-3d+3 = (d-\frac{3}{2})^2 + \frac{3}{4}\geq \frac{3}{4}$, so $-2(d^{2}-3d+3)<0$. Since $d>1$, we have $Q(0)>0$ and $Q(3)<0$.  By continuity, $Q(\lambda')=0$ for some $0<\lambda'<3$. Therefore $\lambda_2(\mathcal{L}_H)=3>\lambda'\geq \lambda_2(\mathcal{L}_G)$, which proves the claim.
\end{proof}

Next we turn to the shape $(2,1^2)$.
We show that no analogous equality occurs in this branch.

\begin{proposition}
\label{prop:211}
Suppose $G$ is a connected $4$-vertex graph, and $H$ is the reduced graph of $G$ at vertex $4$.
Then $\lambda_{\min}(L_G,S^{(2,1^2)})>\lambda_{\min}(L_{H\sqcup\{4\}},S^{(2,1^2)})$.
Hence by the octopus induction scheme of \zcref{sec:octopus}, $\lambda_{\min}(L_G,S^{(2,1^2)})>\lambda_{\min}(L_G,S^{(3,1)})$.
\end{proposition}
\begin{proof}
By the branching rule,
\[
S^{(2,1^2)}\mathord\downarrow_{\mathfrak{S}_3}
\cong S^{(2,1)}\oplus S^{(1^3)}.
\]
The case $n=3$ excludes the sign summand $S^{(1^3)}$ from the minimum eigenspace of $L_{H\sqcup\{4\}}$, so every minimum eigenvector lies in the selected branch
$\mathcal{B}_{(2,1^2)}(4)\cong S^{(2,1)}$.
Proposition~\ref{prop:hook-rigidity}, proved in \zcref{sec:rigidity}, gives
$
\ker\Delta\cap\mathcal{B}_{(2,1^2)}(4)=\{0\}$.
The strict inequality in \eqref{ineq:octo1} follows from Lemma~\ref{lem:kernel}, and the second conclusion follows from the remaining lines of the octopus induction scheme.
\end{proof}

\begin{remark}[Why no second polytabloid matrix is needed]
The two four-vertex sectors, $(2,2)$ and $(2,1^2)$, now have complementary roles.
The explicit $(2,2)$ matrix exhibits the rank drop and exceptional direction leading to the uniformly weighted $4$-cycle.
The $(2,1^2)$ sector serves the complementary purpose of excluding any competing equality mode.
\zcref{sec:poly11} realizes the hook branch $(n-2,1^2)$ as an invariant exterior-square subspace and proves its rigidity for every positive degree.
Repeating a separate hook-polytabloid matrix calculation here would therefore duplicate the general argument without revealing a second exceptional mechanism.
\end{remark}

Together, \zcref{prop:signrep}, \zcref{thm:existunique}, and \zcref{prop:211} establish \zcref{thm:main}, \zcref{item:4cycle}.

\subsection{The case $n=5$}

In this subsection we prove

\begin{proposition}
\label{prop:5}
Let $G$ be a $5$-vertex connected graph. Then for every $\mu\vdash 5$, $\mu \neq (5), (4,1)$, we have
$
\lambda_{\min}(L_G, S^\mu) > \lambda_{\min}(L_G, S^{(4,1)}).
$
\end{proposition}

We begin with the cases settled directly by the four-vertex result. 
If $\mu=(1^5)$, then \zcref{prop:5} holds by \zcref{prop:signrep}.
If $\mu=(2,1^3)$, then by the branching rule (\zcref{prop:branching}), $\mu' = \mu-\Box$ is $(2,1^2)$ or $(1^4)$.
Using \zcref{prop:211} and \zcref{prop:signrep}, respectively, we deduce that $\min_{\mu' = \mu-\Box} \lambda_{\min}(L_H, S^{\mu'}) > \lambda_{\min}(L_H, S^{(3,1)})$, \emph{i.e.,} \eqref{ineq:octo2} is a strict inequality, and hence \zcref{prop:5} holds.

\begin{table}[htp]
\centering
\small
\renewcommand{\arraystretch}{1.2}
\begin{tabularx}{\textwidth}{>{\centering\arraybackslash}p{0.16\textwidth}|*{3}{>{\raggedright\arraybackslash}X}}
$\mu$  & $(3,2)$ & $(3,1^2)$ & $(2^2,1)$ \\ \hline
$\mu' = \mu-\Box$ & $(2,2), (3,1)$ & $(2,1^2), (3,1)$ & $(2,1^2), (2,2)$ \\ \hline
Equality in \eqref{ineq:octo2}? & Yes; attained at $\mu'=(3,1)$. & Yes; attained at $\mu'=(3,1)$. & No, unless $H$ is the cycle graph with uniform weights (\zcref{thm:existunique}); in that case equality is attained at $\mu'=(2,2)$.
\end{tabularx}
\caption{Status of \eqref{ineq:octo2} for the irreducible representations $\mu \vdash 5$ of interest}
\label{table:5status}
\end{table}

So it remains to check $\mu = (3,2), (3,1^2), (2^2,1)$. \zcref{table:5status} gives a summary of their restrictions to $\mathfrak{S}_{n-1}$.
Note that when $H$ is the $4$-cycle with uniform weights and $\mu=(3,2)$, $\min_{\mu'=\mu-\Box}\lambda_{\min}(L_H, S^{\mu'})$ is attained at both $\mu'=(2,2)$ and $\mu'=(3,1)$.

The next two lemmas explain why this exception for $H$ does not obstruct our inductive argument from $n=4$ to $n=5$.

\begin{lemma}
\label{lem:5reducedcycle}
Let $G$ be a connected weighted graph on $\{1,2,3,4,5\}$, with the convention that $c_{ij}=0$ for a missing edge, and let $H$ be its reduction at vertex $5$.  The following are equivalent:
\begin{enumerate}[label=(\alph*)]
\item \label{item:5red1} The reduced graph $H$ is the $4$-cycle
$1\! -\!2\! -\!3\! -\!4\! -\!1$ with a common edge weight $w>0$.
\item \label{item:5red2} After relabeling $1,2,3,4$ by a symmetry of the $4$-cycle (a dihedral relabeling), there exist
$q,r\geq0$ with $q+r>0$ such that
\begin{align*}
&c_{12}=c_{23}=c_{14}=w,\qquad
c_{35}=q,\qquad c_{45}=r,\qquad
c_{34}=w-\frac{qr}{q+r}\geq0,\\
&c_{13}=c_{24}=c_{15}=c_{25}=0.
\end{align*}
Equivalently, the support graph of $G$ is contained in the following schematic.  The dashed edges may be absent, but at least one of $35$ and $45$ is present.
\begin{center}
\GraphFiveVertexException
\end{center}
\end{enumerate}
In particular, in the exceptional family the distinguished vertex $5$ has degree $1$ or $2$ in the support graph.
\end{lemma}

\begin{proof}
Assume first \zcref{item:5red1}.  Since all terms in \eqref{eq:redweights} are nonnegative, the two missing diagonals of the reduced cycle give
\[
0=\widetilde c_{13}=c_{13}+\frac{c_{15}c_{35}}{s},
\qquad
0=\widetilde c_{24}=c_{24}+\frac{c_{25}c_{45}}{s},
\qquad
s:=\sum_{i=1}^4c_{i5}>0.
\]
Hence $c_{13}=c_{24}=0$, $c_{15}c_{35}=0$, and $c_{25}c_{45}=0$.  Thus the positive neighbors of vertex $5$ contain neither opposite pair $\{1,3\}$ nor $\{2,4\}$ of the cycle.  They are therefore contained in one adjacent pair.  After a dihedral relabeling, we may take that pair to be $\{3,4\}$, so $c_{15}=c_{25}=0$.  Put $q=c_{35}$ and $r=c_{45}$; then $q,r\geq0$ and $q+r=s>0$.

The three cycle edges not joining the two possible positive neighbors of vertex $5$ receive no Schur correction, and hence
\[
c_{12}=\widetilde c_{12}=w,\qquad
c_{23}=\widetilde c_{23}=w,\qquad
c_{14}=\widetilde c_{14}=w.
\]
For the remaining cycle edge, \eqref{eq:redweights} gives
\[
w=\widetilde c_{34}=c_{34}+\frac{qr}{q+r},
\]
so $c_{34}=w-qr/(q+r)\geq0$.  This proves \zcref{item:5red2}.

Conversely, suppose \zcref{item:5red2} holds.  Here $s=q+r$, and direct substitution in \eqref{eq:redweights} gives
\[
\widetilde c_{12}=\widetilde c_{23}=\widetilde c_{34}=\widetilde c_{14}=w,
\qquad
\widetilde c_{13}=\widetilde c_{24}=0.
\]
Thus $H$ is the uniformly weighted $4$-cycle, proving \zcref{item:5red1}.
\end{proof}

\begin{lemma}
\label{lem:exceptionH}
Let $G$ belong to the exceptional family in \zcref{lem:5reducedcycle}, \zcref{item:5red2}, and let $H$ be the reduced graph at vertex $5$. Then $\lambda_2(\mathcal{L}_H)>\lambda_2(\mathcal{L}_G)$, \emph{i.e.,} \eqref{ineq:Schur} fails to saturate to equality.
Hence by the octopus induction scheme of \zcref{sec:octopus}, $\lambda_{\min}(L_G, S^\mu) > \lambda_{\min}(L_G, S^{(4,1)})$ for every $\mu \vdash 5$, $\mu\neq (5), (4,1)$.
\end{lemma}
\begin{proof}
Divide every edge weight by the common reduced-cycle weight $w$; this common rescaling preserves all strict spectral comparisons.  Reuse $c_{ij}$ for the rescaled weights, and set $q=c_{35}$ and $r=c_{45}$.  Then
$q,r\geq0$, $q+r>0$, and $1-qr/(q+r)\geq0$, while
${c}_{12}={c}_{23}={c}_{14}=1$ and ${c}_{34}=1-\frac{qr}{q+r}$. 
We already found that $\lambda_2(\mathcal{L}_H)=2$ in \zcref{ex:4cycle}. 
Meanwhile
\[
\mathcal{L}_G
=
\begin{bmatrix}
2 & -1 & 0 & -1 & 0\\
-1 & 2 & -1 & 0 & 0 \\
0 & -1 &2+\frac{q^2}{q+r} & -1+\frac{qr}{q+r} & -q \\
-1 & 0 & -1+\frac{qr}{q+r} & 2+\frac{r^2}{q+r} & -r \\
0 & 0 & -q & -r & q+r
\end{bmatrix}.
\]
A direct determinant calculation gives
$\det(\lambda I-\mathcal{L}_G)=\lambda(\lambda-2)P(\lambda)$, where
\[
P(\lambda) = \lambda^3 - \frac{2(q^2+qr+r^2+3q+3r)}{q+r} \lambda^2 + \frac{2(5q^2+7qr+5r^2+4q+4r)}{q+r}\lambda - 10(q+r).
\]
Besides $0$ and $2$, the other eigenvalues of $\mathcal{L}_G$ are the three roots of the cubic polynomial $P$.
Now $P(0)= -10(q+r)$ and $P(2)= \frac{2(q^2+r^2)}{q+r}$. Since $G$ is connected, at least one of $q$ and $r$ must be positive, so $P(0)<0$ and $P(2)>0$. By continuity of $P$, the intermediate value theorem gives a root $\lambda_*\in(0,2)$.  Hence $\lambda_2(\mathcal L_G)\leq\lambda_*<2=\lambda_2(\mathcal L_H)$.
\end{proof}

\begin{remark}[Continuation of \zcref{rem:warning}]
\label{rem:subtlewarning}
We give a second example illustrating \zcref{rem:warning}.
Let $G$ be as in \zcref{lem:5reducedcycle}, \zcref{item:5red2}, and set $c_{12}=c_{23}=c_{14}=c_{34}=c_{45}=1$ and $c_{35}=0$.
(Thus $G$ is the $4$-cycle together with the additional edge $45$, and all present edges have weight $1$.)
Let $\boldsymbol{w}=[1,0,-1,0]^{\mathsf T} \in \mathbb{R}^4$.
Then $\boldsymbol{w}$ is a second eigenvector of $\mathcal{L}_H$ (with eigenvalue $2$), $\sum_{i=1}^4 c_{i5} \boldsymbol{w}(i) = 0$, and furthermore $\left[\begin{matrix} \boldsymbol{w}\\0\end{matrix}\right] \in \mathbb{R}^5$ is an eigenvector of $\mathcal{L}_G$ with eigenvalue $2$.
But the preceding proof of \zcref{lem:exceptionH} shows that $\lambda_2(\mathcal{L}_G)<2$.
\end{remark}

The exceptional reduction is now isolated: if the chosen removed vertex produces a uniformly weighted $4$-cycle, \zcref{lem:exceptionH} makes \eqref{ineq:Schur} strict.  Otherwise, \zcref{table:5status} leaves only $\mu=(3,2)$ and $\mu=(3,1^2)$, with the minimum attained in the $S^{(3,1)}$ branch $\mathcal{B}_\mu(5)$.

The remaining argument uses the next three propositions and chooses the removed vertex according to the support graph.  If the maximum degree is at least $3$, we relabel a vertex of degree at least $3$ as vertex $5$; by the final sentence of \zcref{lem:5reducedcycle}, this choice cannot belong to the exceptional family.  If the maximum degree is at most $2$, connectedness makes the support graph a path or a cycle.  In the path case we choose an endpoint as vertex $5$.  In the cycle case we choose a cycle vertex as vertex $5$ and then split according to whether its reduction is the exceptional uniformly weighted $4$-cycle.

We now prove the three propositions.  The first is the five-vertex application of the branchwise rigidity statements proved in \zcref{sec:rigidity}.

\begin{proposition}
\label{prop:5octo1}
Let $G$ be a connected $5$-vertex graph different from those defined in \zcref{lem:5reducedcycle}, \zcref{item:5red2}, and $H$ be the reduced graph of $G$ at vertex $5$.
Then $\lambda_{\min}(L_G, S^\mu) > \lambda_{\min}(L_{H\sqcup\{5\}},S^\mu)$ (\emph{i.e.,\@} \eqref{ineq:octo1} is a strict inequality) holds in the following scenarios:
\begin{itemize}
\item Vertex $5$ of $G$ has degree $\geq 3$, and $\mu=(3,2)$.
\item $\mu=(3,1^2)$.
\end{itemize}
\end{proposition}
\begin{proof}
Suppose vertex $5$ has degree $\geq3$ and $\mu=(3,2)$.
Because $H$ is not the uniformly weighted $4$-cycle, \zcref{thm:existunique} and the branching rule place every minimum eigenvector of $L_{H\sqcup\{5\}}|_{S^{(3,2)}}$ in $\mathcal{B}_{(3,2)}(5)$.
By Proposition~\ref{prop:two-row-rigidity}, proved in \zcref{sec:rigidity},
$
\ker\Delta\cap\mathcal{B}_{(3,2)}(5)=\{0\}
$,
so Lemma~\ref{lem:kernel} implies strict inequality in \eqref{ineq:octo1}.

Next suppose $\mu=(3,1^2)$.
Proposition~\ref{prop:211} gives strictness in the $S^{(2,1^2)}$ summand of the restriction of $S^{(3,1^2)}$, so every minimum eigenvector of $L_{H\sqcup\{5\}}|_{S^{(3,1^2)}}$ lies in $\mathcal{B}_{(3,1^2)}(5)$.
By Proposition~\ref{prop:hook-rigidity}, proved in \zcref{sec:rigidity},
$
\ker\Delta\cap\mathcal{B}_{(3,1^2)}(5)=\{0\}$.
The desired strict inequality follows from Lemma~\ref{lem:kernel}.
\end{proof}

We next record the zero-propagation lemma needed for paths and cycles.

\begin{lemma}
\label{lem:vanishcycle}
Let $G$ be a weighted $n$-cycle or weighted $n$-path with positive edge weights. Then no nonzero eigenvector of the graph Laplacian $\mathcal{L}_G$ can vanish at two consecutive vertices.
\end{lemma}

\begin{proof}
Let $\boldsymbol{w}$ be an eigenvector and suppose it vanishes at adjacent vertices $x$ and $y$.  Evaluate the eigenvalue equation at $x$.  Every term incident to a zero neighbor vanishes except possibly the term from the other neighbor of $x$ along the path or cycle; because that edge has positive weight, the coordinate at that neighbor must also be zero.  The same argument at $y$ propagates the zero in the opposite direction whenever such a neighbor exists.  Iterating along the connected path or cycle forces every coordinate of $\boldsymbol{w}$ to vanish, a contradiction.
\end{proof}

\begin{proposition}
\label{prop:pathreduced}
Suppose $G$ is the $n$-path with vertices labeled $1$ through $n$ along the path, and $H$ is the reduced graph of $G$ at vertex $n$. Then $\lambda_2(\mathcal{L}_H) > \lambda_2(\mathcal{L}_G)$, \emph{i.e.,} \eqref{ineq:Schur} is a strict inequality.
\end{proposition}
\begin{proof}
Suppose on the contrary that $\lambda_2(\mathcal{L}_H) = \lambda_2(\mathcal{L}_G)$. 
By \zcref{prop:eigcomp}, there exists a nonzero $\boldsymbol{w} = [\boldsymbol{w}(1), \cdots, \boldsymbol{w}(n-1)]^{\mathsf T}\in \mathbb{R}^{n-1}$ which satisfies $\boldsymbol{w}\perp \boldsymbol{1}^{(n-1)}$, $c_{n-1,n} \boldsymbol{w}(n-1)=0$, and that $\left[\begin{matrix}\boldsymbol{w}\\0\end{matrix}\right] \in \mathbb{R}^n$ is a second eigenvector of $\mathcal{L}_G$. 
Observe that since $c_{n-1,n}>0$, it must be that $\boldsymbol{w}(n-1)=0$.
Thus the vector $\left[\begin{matrix}\boldsymbol{w}\\0\end{matrix}\right]$ takes value $0$ at two consecutive vertices along a path of $G$, and by \zcref{lem:vanishcycle}, it cannot be an eigenvector of $\mathcal{L}_G$. 
We thus arrive at a contradiction.
\end{proof}

Our third proposition settles the last remaining case for $n=5$.

\begin{proposition}
\label{prop:5cycle}
Suppose $G$ is a weighted $5$-cycle whose reduction $H$ at vertex $5$ is not the uniformly weighted $4$-cycle; equivalently, exclude the family in \zcref{lem:5reducedcycle}, part~\ref{item:5red2}.  Label the vertices of $G$ in cyclic order as $5,2,3,4,1$.
For $\mu=(3,2)$, the inequality \eqref{ineq:octo1} is strict.
\end{proposition}
\GraphReductionDisplay

\begin{proof}
Suppose, to the contrary, that equality holds in \eqref{ineq:octo1}.
Since $H$ is not the uniformly weighted $4$-cycle, the results of \zcref{sec:n4} and the branching rule imply that every minimum eigenvector of $L_{H\sqcup\{5\}}|_{S^{(3,2)}}$ lies in $\mathcal{B}_{(3,2)}(5)$.
By Lemma~\ref{lem:kernel}, there is therefore a nonzero vector in
$\ker\Delta\cap\mathcal{B}_{(3,2)}(5)$ which is a minimum eigenvector of $L_{H\sqcup\{5\}}$.

Set $a=c_{15}$ and $b=c_{25}$.
By Proposition~\ref{prop:two-row-low-degree} and the branch realization of Proposition~\ref{prop:2branchmodel}, the corresponding graph-Laplacian eigenvector on $H$ is, up to a nonzero scalar multiple,
\[
\boldsymbol{u}=\big[3a+b,\,-(a+3b),\,-(a-b),\,-(a-b)\big]^{\mathsf T}.
\]
The vertices $3$ and $4$ have the same coordinate, whereas
\[
(\mathcal{L}_H\boldsymbol{u})(3)=4b\,c_{23}>0,
\qquad
(\mathcal{L}_H\boldsymbol{u})(4)=-4a\,c_{14}<0.
\]
These two values would have to be equal if $\boldsymbol{u}$ were an eigenvector, because $\boldsymbol{u}(3)=\boldsymbol{u}(4)$.
This contradiction proves that \eqref{ineq:octo1} is strict.
\end{proof}

\begin{proof}[Proof of \zcref{prop:5}]
The partitions $(1^5)$ and $(2,1^3)$ were settled before \zcref{table:5status}.  It remains to consider $(3,2)$, $(3,1^2)$, and $(2^2,1)$.

If the support graph has a vertex of degree at least $3$, choose such a vertex as vertex $5$.  This reduction is not exceptional by \zcref{lem:5reducedcycle}.  Hence \eqref{ineq:octo2} is strict for $\mu=(2^2,1)$ by \zcref{table:5status}, whereas \zcref{prop:5octo1} makes \eqref{ineq:octo1} strict for $\mu=(3,2)$ and $\mu=(3,1^2)$.

Suppose instead that the support graph has maximum degree at most $2$.  If it is a path, choose an endpoint as vertex $5$; then \zcref{prop:pathreduced} makes \eqref{ineq:Schur} strict for every remaining partition.  If it is a cycle, choose any cycle vertex as vertex $5$ and label the cycle as in \zcref{prop:5cycle}.  When the reduction is the uniformly weighted $4$-cycle, \zcref{lem:exceptionH} applies.  Otherwise, \eqref{ineq:octo2} is strict for $\mu=(2^2,1)$ by \zcref{table:5status}, \zcref{prop:5octo1} handles $\mu=(3,1^2)$, and \zcref{prop:5cycle} handles $\mu=(3,2)$.  These alternatives exhaust all connected support graphs on five vertices, and thus prove \zcref{prop:5}.
\end{proof}

\subsection{Induction for $n\geq 6$}

For $n\geq6$, the branchwise rigidity results of \zcref{sec:rigidity} handle vertices of degree at least $3$, while the path and cycle arguments handle graphs of maximum degree at most $2$.  We now assemble these cases into the induction.

\begin{proof}[Proof of \zcref{thm:main}]
It remains to prove \zcref{item:not4cycle}.
Suppose the induction hypothesis holds for some $n-1\geq5$: for every connected graph $H$ on $n-1$ vertices and every $\mu'\vdash n-1$, $\mu'\neq(n-1),(n-2,1)$, one has
$
\lambda_{\min}(L_H,S^{\mu'})>\lambda_{\min}(L_H,S^{(n-2,1)})$.
The base case $n-1=5$ is \zcref{prop:5}.

Fix a connected $n$-vertex graph $G$ and choose a vertex labeled $n$; let $H$ be the reduced graph at $n$.
By the induction hypothesis, \eqref{ineq:octo2} is strict unless $\mu=(n-2,2)$ or $\mu=(n-2,1^2)$, and in either remaining case every minimum eigenvector of $L_{H\sqcup\{n\}}|_{S^\mu}$ lies in $\mathcal{B}_\mu(n)$.
We show that one of the other inequalities in the octopus scheme is strict.

\begin{itemize}[wide]
\item Suppose $\mu=(n-2,2)$ and $G$ has maximum degree at most $2$.
If $G$ is a path, choose an endpoint to be vertex $n$; then \zcref{prop:pathreduced} makes \eqref{ineq:Schur} strict.
If $G$ is a cycle, label the two neighbors of $n$ by $1$ and $2$, in cyclic order, so that the vertices $3,4,\ldots,n-1$ are consecutive in $H$.
If equality held in \eqref{ineq:octo1}, then Lemma~\ref{lem:kernel}, Proposition~\ref{prop:2branchmodel}, and Proposition~\ref{prop:two-row-low-degree} would produce a second eigenvector $\boldsymbol{u}$ of $\mathcal{L}_H$ which is constant on $3,4,\ldots,n-1$.
Since $n\geq6$, the vertices $3,4,5$ are consecutive.
The eigenvalue equation at vertex $4$ gives $\lambda_2(\mathcal{L}_H)\boldsymbol{u}(4)=0$, so this common value is $0$.
Thus $\boldsymbol{u}$ vanishes at two consecutive vertices, contradicting \zcref{lem:vanishcycle}.

\item Suppose $\mu=(n-2,2)$ and $G$ has a vertex of degree at least $3$.
Choose such a vertex to be $n$.
Then Proposition~\ref{prop:two-row-rigidity} gives
$
\ker\Delta\cap\mathcal{B}_{(n-2,2)}(n)=\{0\}$,
so \eqref{ineq:octo1} is strict by Lemma~\ref{lem:kernel}.

\item Finally, suppose $\mu=(n-2,1^2)$.
By Proposition~\ref{prop:hook-rigidity},
$
\ker\Delta\cap\mathcal{B}_{(n-2,1^2)}(n)=\{0\}$,
and again \eqref{ineq:octo1} is strict.
\end{itemize}

Consequently
$
\lambda_{\min}(L_G,S^\mu)>\lambda_{\min}(L_G,S^{(n-1,1)})
$
for every $\mu\vdash n$, $\mu\neq(n),(n-1,1)$.
\end{proof}

\section{Branchwise octopus rigidity}
\label{sec:rigidity}

By \zcref{lem:kernel}, the exact obstruction to strictness in \eqref{ineq:octo1} is
$\ker\Delta\cap E_{\min}$, where $E_{\min}$ is the minimum eigenspace of
$L_{H\sqcup\{n\}}|_{S^\mu}$.  Under the inductive strictness hypothesis,
$E_{\min}\subseteq\mathcal{B}_\mu(n)$, so it suffices to compute
$\ker\Delta\cap\mathcal{B}_\mu(n)$.  We call its elements
\emph{branch-kernel vectors}.

The two-row branch $\mu=(n-2,2)$ is rigid when the removed vertex has degree at least $3$ and has a one-dimensional branch kernel in degrees $1$ and $2$.  The hook branch $\mu=(n-2,1^2)$ is rigid at every positive degree.  The direct $(2,2)$ calculation in \zcref{sec:n4} is recovered as the smallest two-row case.

Throughout this section assume $n\geq4$ and set
\[
m=n-1,\qquad c_i=c_{in}\quad(1\leq i\leq m),\qquad s=\sum_{i=1}^m c_i>0.
\]
An index $i$ with $c_i>0$ is an \emph{arm} of the removed vertex $n$; an index with $c_i=0$ is a \emph{nonarm}.  Thus the number of arms is $\deg_G(n)$ in the support graph defined in \zcref{sec:setup}.
To keep the normalization distinct from that of \eqref{eq:octopusop}, write
\begin{align}
\label{eq:scaledDelta}
\widehat{\Delta}
:=s\Delta
=s\sum_{i=1}^m c_i(\operatorname{Id}-(i,n))
-\sum_{1\leq i<j\leq m}c_ic_j(\operatorname{Id}-(i,j)).
\end{align}
Since $s>0$, one has $\ker\widehat{\Delta}=\ker\Delta$ on every representation.
Finally, define the mean-zero coordinate space
\[
U_m:=\left\{\boldsymbol{u}=(u_1,\ldots,u_m)\in\mathbb{R}^m:\sum_{i=1}^m u_i=0\right\},
\]
and observe that $U_m\cong S^{(n-2,1)}$ as an $\mathfrak{S}_m$-module.

\subsection{The $(n-2,2)$ branch}

Identify the permutation module $M^{(n-2,2)}$ with the space of functions $F$ on the $2$-subsets of $[n]$.
For such an $F$, set
\[
(\partial F)(i):=\sum_{j\in [n],~j\neq i}F_{\{i,j\}},\qquad i\in[n].
\]
The map $\partial:M^{(n-2,2)}\to M^{(n-1,1)}$ is $\mathfrak{S}_n$-equivariant.
\begin{lemma}
\label{lem:kerpartial}
The $\mathfrak{S}_n$-submodule
\[
\ker\partial
=\left\{F:\sum_{j\in [n],~j\neq i}F_{\{i,j\}}=0\text{ for every }i\in[n]\right\}
\]
is isomorphic to $S^{(n-2,2)}$.
\end{lemma}
\begin{proof}
By Young's rule,
$
M^{(n-2,2)}\cong S^{(n)}\oplus S^{(n-1,1)}\oplus S^{(n-2,2)}
$.
The map $\partial$ is nonzero on the trivial summand $S^{(n)}$.
For the standard summand, identify $S^{(n-1,1)}$ with
$V_n=\{\boldsymbol{v}\in\mathbb{R}^n:\boldsymbol{v}\perp\boldsymbol{1}^{(n)}\}$ and define
\[
\iota:V_n\longrightarrow M^{(n-2,2)},
\qquad
(\iota\boldsymbol{v})_{\{i,j\}}=v_i+v_j.
\]
The map $\iota$ is $\mathfrak{S}_n$-equivariant, and direct summation gives
$\partial\iota=(n-2)\operatorname{Id}_{V_n}$.  Hence $\iota$ is injective and $\partial$ maps its image isomorphically onto the standard summand of $M^{(n-1,1)}$.
On the trivial summand, the constant function $F_{\{i,j\}}=1$ is sent to the nonzero constant vector $(n-1)\boldsymbol{1}^{(n)}$, so $\partial$ also maps the trivial summand onto the trivial summand of the target.
Thus $\partial$ is surjective onto $M^{(n-1,1)}\cong S^{(n)}\oplus S^{(n-1,1)}$.
No copy of $S^{(n-2,2)}$ occurs in the target, and the multiplicity-free decomposition above therefore gives $\ker\partial\cong S^{(n-2,2)}$.
\end{proof}

Henceforth, realize $S^{(n-2,2)}$ as $\ker\partial$ and
$\mathcal{B}_{(n-2,2)}(n)$ as its unique $\mathfrak{S}_m$-submodule isomorphic to $S^{(n-2,1)}$.
The next proposition gives this branch explicitly.

\begin{proposition}
\label{prop:2branchmodel}
Let $\Phi_2:U_m\to\ker\partial$ be given by
\ifhtml
\begin{equation}(\Phi_2\boldsymbol{u})_{\{i,j\}}=u_i+u_j,\qquad1\leq i<j\leq m\label{eq:Phi2leaf}\end{equation}
\begin{equation}(\Phi_2\boldsymbol{u})_{\{i,n\}}=-(n-3)u_i,\qquad1\leq i\leq m\label{eq:Phi2root}\end{equation}
\else
\begin{align}
(\Phi_2\boldsymbol{u})_{\{i,j\}}&=u_i+u_j, &&1\leq i<j\leq m,
\label{eq:Phi2leaf}\\
(\Phi_2\boldsymbol{u})_{\{i,n\}}&=-(n-3)u_i, &&1\leq i\leq m.
\label{eq:Phi2root}
\end{align}
\fi
Then $\Phi_2$ is an $\mathfrak{S}_m$-equivariant isomorphism from $U_m$ onto $\mathcal{B}_{(n-2,2)}(n)$.
\end{proposition}

\begin{proof}
For $\boldsymbol{u}\in U_m$, direct summation in \eqref{eq:Phi2leaf}--\eqref{eq:Phi2root} gives $\partial(\Phi_2\boldsymbol{u})=0$.
Furthermore, $\Phi_2$ is $\mathfrak{S}_m$-equivariant and injective, so its image is an $\mathfrak{S}_m$-submodule isomorphic to $S^{(n-2,1)}$.
For $n\geq5$, the multiplicity-free branching rule gives
\[
S^{(n-2,2)}\mathord\downarrow_{\mathfrak{S}_m}
\cong S^{(n-3,2)}\oplus S^{(n-2,1)},
\]
and hence the image is exactly $\mathcal{B}_{(n-2,2)}(n)$.
For $n=4$, one instead has
\[
S^{(2,2)}\mathord\downarrow_{\mathfrak{S}_3}\cong S^{(2,1)},
\]
and both $U_3$ and $\ker\partial$ have dimension $2$; hence the injective map $\Phi_2:U_3\to\ker\partial$ is onto.
\end{proof}

The map $\Phi_2$ converts the branchwise kernel problem into the following linear system on $U_m$, expressed in coordinates adapted to the arms of the removed vertex.

\begin{lemma}
\label{lem:DeltaPhi2}
For $\boldsymbol{u}\in U_m$, set
\[
q:=\sum_{k=1}^m c_ku_k,
\qquad
y_i:=(n-2)su_i+q\quad(1\leq i\leq m).
\]
Then
\ifhtml
\begin{equation}(\widehat{\Delta}\Phi_2\boldsymbol{u})_{\{i,j\}}=c_jy_i+c_iy_j,\qquad1\leq i<j\leq m\label{eq:DeltaPhi2leaf}\end{equation}
\begin{equation}(\widehat{\Delta}\Phi_2\boldsymbol{u})_{\{i,n\}}=(2c_i-s)y_i-(n-1)c_iq,\qquad1\leq i\leq m\label{eq:DeltaPhi2root}\end{equation}
\else
\begin{align}
(\widehat{\Delta}\Phi_2\boldsymbol{u})_{\{i,j\}}&=c_jy_i+c_iy_j,
&&1\leq i<j\leq m,
\label{eq:DeltaPhi2leaf}\\
(\widehat{\Delta}\Phi_2\boldsymbol{u})_{\{i,n\}}&=(2c_i-s)y_i-(n-1)c_iq,
&&1\leq i\leq m.
\label{eq:DeltaPhi2root}
\end{align}
\fi
\end{lemma}

\begin{proof}
Fix $i<j\leq m$ and recall \eqref{eq:scaledDelta}.
In the first sum in \eqref{eq:scaledDelta}, only the transpositions $(i,n)$ and $(j,n)$ change the pair $\{i,j\}$, and their contribution to $(\widehat{\Delta}\Phi_2\boldsymbol{u})_{\{i,j\}}$ is
\[
s c_i\big(u_i+(n-2)u_j\big)+s c_j\big(u_j+(n-2)u_i\big).
\]
In the second sum in \eqref{eq:scaledDelta}, only transpositions containing $i$ or $j$ contribute, giving
\[
-c_i\sum_{k\neq i,j}c_k(u_i-u_k)
-c_j\sum_{k\neq i,j}c_k(u_j-u_k).
\]
Using $q=\sum_kc_ku_k$ and simplifying yields \eqref{eq:DeltaPhi2leaf}.
The same direct calculation at the pair $\{i,n\}$ gives
\[
(\widehat{\Delta}\Phi_2\boldsymbol{u})_{\{i,n\}}
=(n-2)s(2c_i-s)u_i-\big(s+(n-3)c_i\big)q,
\]
which is equivalent to \eqref{eq:DeltaPhi2root}.
\end{proof}

We now solve $\widehat{\Delta}\Phi_2\boldsymbol{u}=0$.

\begin{proposition}[Rigidity of the $(n-2,2)$ branch when degree $\geq3$]
\label{prop:two-row-rigidity}
If vertex $n$ has at least three arms, then
\[
\ker\Delta\cap\mathcal{B}_{(n-2,2)}(n)
=
\ker\widehat{\Delta}\cap\mathcal{B}_{(n-2,2)}(n)
=\{0\}.
\]
\end{proposition}

\begin{proof}
Let $\Phi_2\boldsymbol{u}\in\ker\widehat{\Delta}\cap\mathcal{B}_{(n-2,2)}(n)$, and use the notation of Lemma~\ref{lem:DeltaPhi2}.
Choose three indices $i,j,k$ for which $c_i,c_j,c_k>0$.
Equation \eqref{eq:DeltaPhi2leaf} gives
\[
c_jy_i+c_iy_j=0,
\qquad
c_ky_i+c_iy_k=0,
\qquad
c_ky_j+c_jy_k=0.
\]
After dividing by the positive products and setting $z_\ell=y_\ell/c_\ell$, these become
$z_i+z_j=z_i+z_k=z_j+z_k=0$.
Hence $y_i=y_j=y_k=0$.
Pairing one of these indices with every other arm $\ell$ in \eqref{eq:DeltaPhi2leaf} gives $y_\ell=0$ for every arm.
For a nonarm $r$, equation \eqref{eq:DeltaPhi2root} reduces to $-sy_r=0$, and hence $y_r=0$ as well.
Thus every $u_\ell$ equals $-q/((n-2)s)$.
Since $\boldsymbol{u}\in U_m$, this common value is zero, so $\boldsymbol{u}=0$ and therefore $\Phi_2\boldsymbol{u}=0$.
The equality of the two displayed kernel intersections follows from $\ker\widehat{\Delta}=\ker\Delta$.
\end{proof}

When vertex $n$ has one or two arms, one scalar branch-kernel direction survives.
Its coordinate vector is constant on every vertex not adjacent to $n$, which is the feature used in the cycle arguments of \zcref{sec:mainproof}, \emph{cf.\@} the proof of \zcref{prop:5cycle}.

\begin{proposition}[The branch-kernel line in the $(n-2,2)$ sector when degree $\leq2$]
\label{prop:two-row-low-degree}
Suppose vertex $n$ has one or two arms.
Relabel the vertices so that
\[
c_1>0,\qquad c_2\geq0,\qquad c_i=0\quad(3\leq i\leq m),
\]
where $c_2=0$ in the degree-$1$ case.
Define $\boldsymbol{u}^*\in U_m$ by
\begin{align}
u^*_1&=(n-2)c_1+c_2,
\notag\\
u^*_2&=-\big(c_1+(n-2)c_2\big),
\label{eq:ustar}\\
u^*_j&=-(c_1-c_2),\qquad 3\leq j\leq m.
\notag
\end{align}
Then
\[
\ker\Delta\cap\mathcal{B}_{(n-2,2)}(n)
=
\ker\widehat{\Delta}\cap\mathcal{B}_{(n-2,2)}(n)
=\mathbb{R}\,\Phi_2\boldsymbol{u}^*.
\]
In particular, every branch-kernel vector corresponds under $\Phi_2$ to a vector that is constant on the vertices not adjacent to $n$.
\end{proposition}

\begin{proof}
Let $\Phi_2\boldsymbol{u}$ belong to the intersection with $\ker\widehat{\Delta}$.
With the notation of Lemma~\ref{lem:DeltaPhi2}, equation \eqref{eq:DeltaPhi2leaf} applied to $\{1,r\}$ gives $y_r=0$ for every $r\geq3$.
The equation for $\{1,2\}$ gives
$
c_2y_1+c_1y_2=0
$,
and \eqref{eq:DeltaPhi2root} at $i=1$ gives
$
(c_1-c_2)y_1-mc_1q=0
$.
Since $c_1>0$, these relations determine $y_2,\ldots,y_m$ and $q$ from the single parameter $y_1$; then
\[
u_i=\frac{y_i-q}{(n-2)s}
\]
determines $\boldsymbol{u}$.
Thus the intersection has dimension at most $1$.

It remains to exhibit a nonzero vector in it.
For $\boldsymbol{u}=\boldsymbol{u}^*$, one computes
\[
q=(n-2)s(c_1-c_2),
\qquad
y_1=(n-2)s(n-1)c_1,
\qquad
y_2=-(n-2)s(n-1)c_2,
\qquad
y_r=0\quad(r\geq3).
\]
Substitution into \eqref{eq:DeltaPhi2leaf} and \eqref{eq:DeltaPhi2root} gives $\widehat{\Delta}\Phi_2\boldsymbol{u}^*=0$.
The vector $\boldsymbol{u}^*$ is nonzero and belongs to $U_m$, so it spans the intersection.
The equality with the $\Delta$-kernel intersection again follows from $\ker\widehat{\Delta}=\ker\Delta$.
\end{proof}

\begin{remark}[The four-vertex dictionary]
We now close the loop with the explicit calculation of \zcref{sec:n4}.
When $n=4$, the restriction $S^{(2,2)}\mathord\downarrow_{\mathfrak{S}_3}\cong S^{(2,1)}$ has only one summand, so
$
\mathcal{B}_{(2,2)}(4)=S^{(2,2)}
$.
The branch model also has a simple dictionary with the polytabloid basis from Lemma~\ref{lem:four-vertex-two-row-kernel}.
If $\boldsymbol{u}=(u_1,u_2,u_3)\in U_3$, comparison of the six two-subset coordinates gives
$
\Phi_2\boldsymbol{u}=-u_2\boldsymbol{e}_{t_2}-u_3\boldsymbol{e}_{t_3}
$.
Consequently, Proposition~\ref{prop:two-row-rigidity} specializes exactly to part~\ref{item:4deg3} of Lemma~\ref{lem:four-vertex-two-row-kernel}.
In the degree-at-most-$2$ case, interchange vertices $1$ and $2$ if necessary so that $c_1>0$, and relabel the remaining vertex so that $c_3=0$.
Then \eqref{eq:ustar} becomes
\[
\boldsymbol{u}^*=(2c_1+c_2,-c_1-2c_2,-c_1+c_2),
\]
and hence
\[
\Phi_2\boldsymbol{u}^*=(c_1+2c_2)\boldsymbol{e}_{t_2}+(c_1-c_2)\boldsymbol{e}_{t_3}.
\]
This is precisely the kernel generator in part~\ref{item:4deg2} of Lemma~\ref{lem:four-vertex-two-row-kernel}.
Thus the direct polytabloid computation discovers and illustrates the exceptional direction, while Propositions~\ref{prop:two-row-rigidity} and~\ref{prop:two-row-low-degree} identify it as the $n=4$ instance of the general degree dichotomy.
\end{remark}

\subsection{The $(n-2,1^2)$ branch}
\label{sec:poly11}

For the second surviving shape, the exterior-square realization makes the selected branch explicit.

Over characteristic zero, the Specht module of the hook $(n-r,1^r)$ is the $r$th exterior power of the standard representation.
We work over $\mathbb{R}$; the identification follows by base change from the usual integral models of the standard and hook Specht modules.
In particular,
\[
S^{(n-2,1^2)}\cong\bigwedge^2 S^{(n-1,1)}.
\]
See, for example, \cite{FultonHarris}*{Exercise 4.6} for this standard exterior-power realization.

Let
\[
V_n:=\left\{\boldsymbol{z}\in\mathbb{R}^n:\sum_{i=1}^n z_i=0\right\}
\]
be the standard representation, and henceforth realize $S^{(n-2,1^2)}$ as $\bigwedge^2V_n$.
Embed $U_m$ in $V_n$ by appending a zero $n$th coordinate, and set
\[
\boldsymbol{h}=(1,\ldots,1,-m)\in V_n.
\]
Since $V_n\mathord\downarrow_{\mathfrak{S}_m}\cong U_m\oplus\mathbb{R}\boldsymbol{h}$, one has
\[
\bigwedge^2V_n \bigg\downarrow_{\mathfrak{S}_m}
\cong \bigwedge^2U_m\oplus(U_m\wedge\boldsymbol{h}).
\]
Under this realization, the second summand is the selected branch $\mathcal{B}_{(n-2,1^2)}(n)$.
Define $\Phi_\wedge:U_m\to\bigwedge^2V_n$ by
\[
\Phi_\wedge\boldsymbol{v}:=\boldsymbol{v}\wedge\boldsymbol{h}.
\]
For a wedge vector $\omega$, write its coordinates antisymmetrically as $\omega_{ji}:=-\omega_{ij}$, and let permutations act by
$(\sigma\omega)_{ij}=\omega_{\sigma^{-1}(i),\sigma^{-1}(j)}$.
Then
\ifhtml
\begin{equation}(\Phi_\wedge\boldsymbol{v})_{ij}=v_i-v_j,\qquad1\leq i<j\leq m\label{eq:PhiWleaf}\end{equation}
\begin{equation}(\Phi_\wedge\boldsymbol{v})_{in}=-m v_i,\qquad1\leq i\leq m\label{eq:PhiWroot}\end{equation}
\else
\begin{align}
(\Phi_\wedge\boldsymbol{v})_{ij}&=v_i-v_j, &&1\leq i<j\leq m,
\label{eq:PhiWleaf}\\
(\Phi_\wedge\boldsymbol{v})_{in}&=-m v_i, &&1\leq i\leq m.
\label{eq:PhiWroot}
\end{align}
\fi
The preceding decomposition implies that $\Phi_\wedge$ is an $\mathfrak{S}_m$-equivariant isomorphism from $U_m$ onto $\mathcal{B}_{(n-2,1^2)}(n)$.

\begin{lemma}
\label{lem:DeltaPhiW}
For $\boldsymbol{v}\in U_m$, set
\[
q:=\sum_{k=1}^m c_kv_k,
\qquad
y_i:=nsv_i-q.
\]
Then, for $1\leq i<j\leq m$,
\begin{align}
\label{eq:DeltaPhiW}
(\widehat{\Delta}\Phi_\wedge\boldsymbol{v})_{ij}=c_jy_i-c_iy_j.
\end{align}
\end{lemma}

\begin{proof}
Fix $i<j\leq m$.
The transpositions $(i,n)$ and $(j,n)$ from the first sum in \eqref{eq:scaledDelta} contribute
\[
s c_i(v_i-nv_j)+s c_j(nv_i-v_j).
\]
In the second sum in \eqref{eq:scaledDelta}, the transposition $(i,j)$ reverses the orientation of the $ij$-coordinate and contributes $-2c_ic_j(v_i-v_j)$.
The remaining transpositions containing $i$ or $j$ contribute
\[
-c_i\sum_{k\neq i,j}c_k(v_i-v_k)
+c_j\sum_{k\neq i,j}c_k(v_j-v_k).
\]
Using $q=\sum_kc_kv_k$ and simplifying gives \eqref{eq:DeltaPhiW}.
\end{proof}

\begin{proposition}[Rigidity of the $(n-2,1^2)$ branch]
\label{prop:hook-rigidity}
Under the standing assumption $s>0$,
\[
\ker\Delta\cap\mathcal{B}_{(n-2,1^2)}(n)
=
\ker\widehat{\Delta}\cap\mathcal{B}_{(n-2,1^2)}(n)
=\{0\}.
\]
\end{proposition}

\begin{proof}
Let $\Phi_\wedge\boldsymbol{v}$ belong to the intersection with $\ker\widehat{\Delta}$, and use the notation of Lemma~\ref{lem:DeltaPhiW}.
Equation \eqref{eq:DeltaPhiW} implies
\[
c_jy_i=c_iy_j\qquad(1\leq i<j\leq m).
\]
Choose an arm $a$ and set $\rho:=y_a/c_a$.
Pairing $a$ with every other index shows that $y_i=\rho c_i$ for every $i$, including nonarms.
Summing $y_i=nsv_i-q$ and using $\sum_i v_i=0$ gives
\begin{align}
\rho s=-mq.
\label{eq:hookrelation1}
\end{align}
On the other hand,
\[
v_i=\frac{\rho c_i+q}{ns},
\]
so the definition of $q$ gives
\begin{align}
msq=\rho\sum_{i=1}^m c_i^2.
\label{eq:hookrelation2}
\end{align}
Combining \eqref{eq:hookrelation1} and \eqref{eq:hookrelation2} yields
\[
mq\left(s+\frac{\sum_i c_i^2}{s}\right)=0.
\]
Thus $q=0$, then $\rho=0$, and finally $\boldsymbol{v}=0$.
Therefore $\Phi_\wedge\boldsymbol{v}=0$.
The equality of the two displayed kernel intersections follows from $\ker\widehat{\Delta}=\ker\Delta$.
\end{proof}


\begin{bibdiv}
\begin{biblist}



\bib{AKP25}{article}{
   author={Alon, Gil},
   author={Kozma, Gady},
   author={Puder, Doron},
   title={{On the Aldous-Caputo spectral gap conjecture for hypergraphs}},
   journal={Math. Proc. Cambridge Philos. Soc.},
   volume={179},
   date={2025},
   number={2},
   pages={259--298},
   issn={0305-0041},
   review={\MR{4945969}},
   doi={10.1017/S0305004125000179},
}

\bib{CLR}{article}{
   author={Caputo, Pietro},
   author={Liggett, Thomas M.},
   author={Richthammer, Thomas},
   title={{Proof of Aldous' spectral gap conjecture}},
   journal={J. Amer. Math. Soc.},
   volume={23},
   date={2010},
   number={3},
   pages={831--851},
   issn={0894-0347},
   review={\MR{2629990}},
   doi={10.1090/S0894-0347-10-00659-4},
}

\bib{Cesi}{article}{
   author={Cesi, Filippo},
   title={{A few remarks on the octopus inequality and Aldous' spectral gap
   conjecture}},
   journal={Comm. Algebra},
   volume={44},
   date={2016},
   number={1},
   pages={279--302},
   issn={0092-7872},
   review={\MR{3413687}},
   doi={10.1080/00927872.2014.975349},
}

\bib{DS81}{article}{
   author={Diaconis, Persi},
   author={Shahshahani, Mehrdad},
   title={Generating a random permutation with random transpositions},
   journal={Z. Wahrsch. Verw. Gebiete},
   volume={57},
   date={1981},
   number={2},
   pages={159--179},
   issn={0044-3719},
   review={\MR{0626813}},
   doi={10.1007/BF00535487},
}

\bib{Fulton}{book}{
   author={Fulton, William},
   title={Young tableaux},
   series={London Mathematical Society Student Texts},
   volume={35},
   subtitle={With applications to representation theory and geometry},
   publisher={Cambridge University Press, Cambridge},
   date={1997},
   pages={x+260},
   isbn={0-521-56144-2},
   isbn={0-521-56724-6},
   review={\MR{1464693}},
}

\bib{FultonHarris}{book}{
   author={Fulton, William},
   author={Harris, Joe},
   title={Representation theory},
   series={Graduate Texts in Mathematics},
   volume={129},
   publisher={Springer-Verlag, New York},
   date={1991},
   pages={xvi+551},
   isbn={0-387-97527-6},
   isbn={0-387-97495-4},
   review={\MR{1153249}},
   doi={10.1007/978-1-4612-0979-9},
}

\bib{GodsilMeagher}{book}{
   author={Godsil, Chris},
   author={Meagher, Karen},
   title={Erd\H os-Ko-Rado theorems: algebraic approaches},
   series={Cambridge Studies in Advanced Mathematics},
   volume={149},
   publisher={Cambridge University Press, Cambridge},
   date={2016},
   pages={xvi+335},
   isbn={978-1-107-12844-6},
   review={\MR{3497070}},
   doi={10.1017/CBO9781316414958},
}

\bib{HJ}{article}{
   author={Handjani, Shirin},
   author={Jungreis, Douglas},
   title={Rate of convergence for shuffling cards by transpositions},
   journal={J. Theoret. Probab.},
   volume={9},
   date={1996},
   number={4},
   pages={983--993},
   issn={0894-9840},
   review={\MR{1419872}},
   doi={10.1007/BF02214260},
}

\bib{Sagan}{book}{
   author={Sagan, Bruce E.},
   title={{The Symmetric Group}},
   series={Graduate Texts in Mathematics},
   volume={203},
   edition={2},
   subtitle={Representations, combinatorial algorithms, and symmetric functions},
   publisher={Springer-Verlag, New York},
   date={2001},
   pages={xvi+238},
   isbn={0-387-95067-2},
   review={\MR{1824028}},
   doi={10.1007/978-1-4757-6804-6},
}



   
\end{biblist}
\end{bibdiv}
\end{document}